\documentclass[10pt]{amsart}
\usepackage{amsfonts}

\usepackage{ifthen}
\usepackage{amsthm}
\usepackage{amsmath}
\usepackage{mathrsfs}
\usepackage{graphicx}
\usepackage{amscd,amssymb,amsthm}
\usepackage{enumerate}
\usepackage{epstopdf}
\usepackage{mathrsfs}
\usepackage{hyperref}
\usepackage{float}
\usepackage[table,xcdraw]{xcolor}
\usepackage{multirow}
\usepackage{tikz} 

\usepackage{enumitem}
\usepackage{graphicx}
\usepackage{subcaption}

\usepackage{subcaption}

\usepackage{xcolor}

\usepackage{algorithm}
\usepackage{algpseudocode}
\usepackage{microtype} 

\newcounter{minutes}
\divide\time by 60
\newcounter{hours}
\multiply\time by 60 \addtocounter{minutes}{-\time}

\title[The NWU ageing property of the gamma-gamma distribution ]{The NWU ageing property of the gamma-gamma distribution}

\author[\'A. Baricz]{\'Arp\'ad Baricz}
\address{Department of Economics, Babe\c{s}-Bolyai University, 400591 Cluj-Napoca, Romania and Institute of Applied Mathematics, \'Obuda University, 1034 Budapest, Hungary}
\email{bariczocsi@yahoo.com}

\author[D.K. Prabhu]{Dhivya Prabhu K}
\address{Institute of Applied Mathematics, \'Obuda University, 1034 Budapest, Hungary}
\email{dhivyanslm@gmail.com}

\author[S. Sahoo]{Snehalata Sahoo}
\address{Department of Mathematics,
	Indian Institute of Technology Indore, Indore 453552, India}
\email{sahoosnehalata96@gmail.com}

\author[S. Singh]{Sanjeev Singh}
\address{Department of Mathematics,
	Indian Institute of Technology Indore, Indore 453552, India}
\email{snjvsngh@iiti.ac.in}

\newtheorem{theorem}{Theorem}
\newtheorem{definition}{Definition}
\newtheorem{corollary}{Corollary}
\newtheorem{remark}{Remark}
\newtheorem{lemma}{Lemma}

\begin{document}
	
	\def\thefootnote{}
	\footnotetext{ \texttt{File:~\jobname.tex,
			printed: \number\year-\number\month-\number\day,
			\thehours.\ifnum\theminutes<10{0}\fi\theminutes}
	} \makeatletter\def\thefootnote{\@arabic\c@footnote}\makeatother

	
	
\maketitle

\begin{abstract}
Baricz et al. \cite{BPV11} posed the question of whether the $K$-distribution (or gamma-gamma distribution) satisfies the ageing property known as the new is better than used (NBU) property. In this paper we address this question by establishing that, under certain parameter conditions, the gamma-gamma distribution satisfies the new is worse than used (NWU) property rather than the NBU property. For the remaining parameter ranges, we show that the gamma-gamma distribution is neither NBU nor NWU. The proofs rely on tools from the theory of stochastic orders of random variables. Moreover, by exploiting the fact that the random variable that follows the $K$-distribution is actually the product of two independent gamma random variables, we obtain alternative proofs of our results. Similarly, by exploiting the fact that the random variable that follows the beta-prime distribution is actually the quotient of two independent gamma random variables, we establish analogous NBU/NWU properties for this distribution. We also establish sufficient conditions under which the product of two independent positive random variables possesses the NWU and decreasing failure rate (DFR) properties. On the other hand, we show that the survival function of the $K$-distribution is hyperbolically completely monotone by representing it as the Laplace transform of a random variable belonging to the class of generalized gamma convolutions. As a consequence, we establish some monotonicity properties of the Mills ratio and hazard rate function of the $K$-distribution. These results provide further insight into the ageing and monotonicity properties of the gamma-gamma distribution. In addition, we derive a new integral representation and two new series representations, as well as bounds, for the survival function of the $K$-distribution.
\end{abstract}

\section{\bf Introduction and Motivation}

Let $a,$ $b,$ $\alpha$ and $x$ be positive real numbers. The probability density function of the gamma-gamma distribution \cite{BPV11}, also known as the $K$-distribution \cite{Ef12}, is given by
\begin{equation}\label{pdf-k}
f_{a,b,\alpha}(x)=\frac{2(ab)^{\frac{a+b}{2}}}{\Gamma(a)\Gamma(b)\alpha^{\frac{a+b}{2}}}x^{\frac{a+b}{2}-1}K_{a-b}\left(2\sqrt{\frac{ab}{\alpha}x}\right),
\end{equation}
where $\Gamma$ denotes the Euler gamma function and $K_{\nu}$ denotes the modified Bessel function of the second kind of order $\nu$. We denote the survival function of the $K$-distribution by
\begin{equation}\label{survival introduction}
S_{a,b,\alpha}(x)=\frac{2(ab)^{\frac{a+b}{2}}}{\Gamma(a)\Gamma(b)\alpha^{\frac{a+b}{2}}}\int_x^{\infty}t^{\frac{a+b}{2}-1}
    K_{a-b}\left(2\sqrt{\frac{ab}{\alpha}t}\right){\rm d}t.
\end{equation}
Note that the probability density function of the $K$-distribution has three parameters, and the two parameter $K$-distribution was originally introduced by Jakeman and Pusey to model the non-Rayleigh microwave sea echo \cite{JP76}. Moreover, in \cite{Ba86}, Barakat introduced a three-parameter generalization of the $K$-distribution, and it is worth mentioning that the $K$-distribution has found various applications in atmospheric optics, ultrasonic imaging, and wireless communications \cite{AlY10, Ba86, JT87}.

\begin{definition}
 {\em   Let $X$ be a random variable supported on $(a, b)$ and let $x\mapsto S(x)$ be the corresponding survival function. $X$ is said to have the NWU (NBU) property if
    \begin{equation}\label{nbu_equation}
        \frac{S(x+y)}{S(x)}\geq(\leq) S(y) \quad \mbox{for all} \quad  x, y \in (a, b).
    \end{equation} }
\end{definition}

Now, suppose that the random variable $X$ represent the lifetime of an organism. The NWU (or NBU) property means that the probability of an organism surviving for an additional $y$ units of time, given that it has already survived to age $x$, is greater (or less) than the probability that a newborn organism survives for $y$ units of time \cite{NSB18}. The NBU and NWU properties are important ageing notions in reliability theory, with applications in the analysis of reliability systems and shock models \cite{Ar81,LK07}. Thus, determining whether a given lifetime distribution is NBU, or NWU, or neither is an important question in reliability theory. In \cite{BPV11}, Baricz et al. raised the question of whether the $K$-distribution possesses the NBU property. In \cite{BPV11} the above question was attempted by using the results which connects the log-concave property of the probability density function with the NBU property.

In this paper, we address the question posed by Baricz et al. \cite{BPV11} that whether the $K$-distribution possesses the NBU property. To answer this question, we employ concepts from the theory of \textit{stochastic orders}. We will show that the $K$-distribution possesses the NWU property for certain parameter ranges, whereas, in the remaining parameter range, it possesses neither the NBU nor the NWU property. By exploiting the representation of a random variable following the $K$-distribution as the product of two independent gamma random variables, we obtain alternative proofs of these results. Similarly, by representing a beta-prime random variable as the quotient of two independent gamma random variables, we establish analogous results concerning the NBU/NWU properties of the beta-prime distribution. These results motivate a natural question: under what conditions does the product of two independent positive random variables possess the NBU or NWU property? In this direction, we establish a general result showing that the product of two independent positive random variables possesses the NWU property whenever at least one of the random variables possesses the NWU property.

In one of our alternative proofs of the NWU property of the $K$-distribution, we establish the equivalent result that its failure rate function is decreasing. This naturally raises the question of whether the failure rate function possesses stronger monotonicity properties. For example, is the failure rate function convex or concave? Interestingly, we prove that the failure rate function of the $K$-distribution is a Stieltjes function and hence it is completely monotonic. Consequently, its reciprocal, the Mills ratio, is a complete Bernstein function. These results provide a deeper understanding of the analytical properties of the hazard rate function and the Mills ratio of the $K$-distribution, with implications for its ageing behaviour. We represent the survival function of the $K$-distribution as the Laplace transform of a probability distribution belonging to the class of generalized gamma convolutions. Consequently, we establish that the survival function is hyperbolically completely monotonic.

The rest of the paper is organized as follows. Section 2 presents some preliminary results that will be used throughout the remainder of the paper. Section 3 presents our results concerning the NWU property of the $K$-distribution when $a\leq 1$ or $b\leq 1$. First, using Lemma \ref{concave_theorem_f} and \ref{NBU_theorem_S}, we prove this result under the additional assumption $a\geq b$. The key tool in this proof is the Stieltjes representation of the modified Bessel function of the second kind \cite{Is77}. We then relax the condition $a\geq b$ and show that the $K$-distribution still possesses the NWU property. This proof is based on concepts from the theory of \textit{stochastic orders}. Similarly, we show that when $a,b>1$, the $K$-distribution possesses neither the NBU nor the NWU property. We also provide alternative proofs of the aforementioned results. Finally, in the same section, we establish analogous results for the beta-prime distribution, characterizing its NBU/NWU properties in a manner similar to that of the $K$-distribution. In Section 4, we prove that the survival function of the $K$-distribution is hyperbolically completely monotone and also prove some interesting monotonicity properties of the Mills ratio and hazard rate function of the $K$-distribution. In Section 5, we establish a new integral representation involving the incomplete gamma function for the survival function $S_{a,b,\alpha}(x)$ of the $K$-distribution. Using this integral representation, we prove that $S_{a,b,\alpha}(x)$ is completely monotone when $a\leq 1$ or $b\leq 1$. We also provide an alternative proof of one of the results established earlier. Furthermore, we show that the function $\alpha\mapsto S_{a,b,\alpha}(x)$ is increasing on $(0,\infty)$. In the same section, we establish several bounds for $S_{a,b,\alpha}(x)$. Section 6 presents two series representations for $S_{a,b,\alpha}(x)$. Finally, a summary is given in Section 7.

\section{\bf Preliminary Results}
\setcounter{equation}{0}

By definition, a function $g:(a, b)\subseteq\mathbb{R} \to (0, \infty)$ is said to be {\it log-convex (log-concave)} if the following inequality holds true
\begin{equation}
    g(\lambda x+(1-\lambda)y) \leq (\geq) [g(x)]^{\lambda}[g(y)]^{1-\lambda} \quad \mbox{for all}\quad x, y \in (a, b) \quad \mbox{and} \quad \lambda \in (0, 1).
\end{equation}
Let $f:(a, b) \subseteq \mathbb{R} \to [0, \infty)$ be a differentiable probability density function. The corresponding cumulative distribution function and survival function are defined as
\begin{equation*}
    F(x) = \int_{a}^{x} f(t) {\rm d}t \quad \mbox{and} \quad S(x)=\int_x^b f(t) {\rm d}t,
\end{equation*}
respectively. The term \emph{survival function} arises from reliability theory, where $S$ represents the probability that a organism survives beyond age $x$ \cite{NSB18}.

The following result discusses the log-concave or log-convex properties of the cumulative distribution function $F$ and survival function $S$ by assuming the probability density function $f$ to be log-concave or log-convex. In other words, the following lemma describes how the log-concavity or log-convexity of the probability density function $f$ is inherited by the corresponding cumulative distribution function $F$ and survival function $S$.

\begin{lemma}\label{concave_theorem_f}
Let $f$ be a continuously differentiable probability density function supported on the interval $(a,b)$. Then the following statements are valid.
\begin{enumerate}
    \item[\bf a.] If $f$ is log-concave on $(a,b)$, then both $F$ and $S$ are log-concave on $(a,b)$.

    \item[\bf b.] If $f$ is log-convex on $(a,b)$ and $\lim\limits_{x\to a^{-}}f(x)=0$, then $F$ is log-convex on $(a,b)$.

    \item[\bf c.] If $f$ is log-convex on $(a,b)$ and $\lim\limits_{x\to b^{+}}f(x)=0$, then $S$ is log-convex on $(a,b)$.
\end{enumerate}
\end{lemma}

For a proof of Lemma \ref{concave_theorem_f}, the reader can refer to \cite{An98,BB05,Ba10}. The following result provides a criterion for determining whether $X$ has the NBU or NWU property.

\begin{lemma}\label{NBU_theorem_S}
Let $S$ be the survival function of a random variable $X$ supported on $(a,b)$. Then the following statements hold.
\begin{enumerate}
    \item[\bf a.] If $S$ is log-concave on $(a,b)$, then $X$ has the NBU property.
    \item[\bf b.] If $S$ is log-convex on $(a,b)$, then $X$ has the NWU property.
\end{enumerate}
\end{lemma}

For a proof of Lemma \ref{NBU_theorem_S}, we refer the reader to \cite{An98}. The following lemma establishes the relationship between the monotonicity of the failure rate function and the NBU or NWU properties. The proof of the next lemma easily follows from Lemma \ref{NBU_theorem_S}, for a different proof we refer to \cite[p. 159]{BP75} \label{dfr-nwu}. 

\begin{lemma}
Let $X$ be a nonnegative, absolutely continuous random variable with density function $f$ and survival function $S$. Let $h_X(x)=f(x)/S(x)$ denote its failure rate (or hazard rate) function. Then the following statements are valid.
\begin{enumerate}
    \item[\bf a.] If $h_X$ is decreasing, then $X$ has the NWU property.
    \item[\bf b.] If $h_X$ is increasing, then $X$ has the NBU property.
\end{enumerate}
\end{lemma}

The same conclusions hold if the failure rate function $h_X(x)$ is replaced by $-\log S(x)/x,$ that is, if the function $x\mapsto-\log S(x)/x$ is decreasing (or increasing), then $X$ has the NWU (or NBU) property. 

We next recall some definitions and results from the theory of \textit{stochastic orders} \cite{SS94}.

\begin{definition} 
{\em Let $X$ and $Y$ be two random variables with survival functions $\bar{F}$ and $\bar{G}$, respectively. If $\bar{F}(s)<\bar{G}(s)$, for all $s \in (-\infty, \infty)$, then we say $X$ is smaller than $Y$ in the usual stochastic order. This is denoted by $X \leq_{{\rm st}} Y$.}
\end{definition}

The following remark will be useful in one of the proofs presented in a later part of the paper.

\begin{remark}\label{remark on expectation}
{\em If $X\leq_{st}Y$, then $\mathbb{E}[X]\leq\mathbb{E}[Y],$ provided the expectations exist. This follows from the monotonicity of the expectation with respect to the usual stochastic order, see \cite[p. 8]{SS94}.}
\end{remark}

\begin{definition} 
{\em  Let $X$ and $Y$ be two random continuous variables with probability density functions $f$ and $g$, respectively. If $x \mapsto g(x)/f(x)$ is increasing in the union of supports of $X$ and $Y$, then we say $X$ is smaller than $Y$ in the likelihood ratio order. This is denoted by $X \leq_{{\rm lr}} Y$.}
\end{definition}

It is important to mention here that the likelihood ratio order implies the usual
stochastic order (see for example \cite[Theorems 1.B.1 and 1.C.1]{SS94}), that is
\begin{equation}\label{lr_st}
    X \leq_{\rm lr} Y \quad \text{implies} \quad X \leq_{\rm st} Y.
\end{equation}

The next lemma establishes a connection between stochastic orders and the NBU (or NWU) property (see for example \cite[Theorem 1.A.13]{SS94}). 

\begin{lemma}\label{st_NWU}
Let $X$ be a non-negative random variable. $X$ is NBU (NWU) if and only if
\begin{equation}
X \geq_{{\rm st}} (\leq _{{\rm st}}) \{X-t|X>t\} \quad \text{for all} \quad t>0.
\end{equation}
\end{lemma}

We next recall some definitions and results concerning infinitely divisible distributions and generalized gamma convolutions, along with the concepts of completely monotone and hyperbolically completely monotone functions. For definitions and results concerning infinitely divisible distributions and generalized gamma convolutions, we refer the reader to \cite{Fe66, Bo92}.

A probability distribution $F$ supported on the interval $(0, \infty)$ is {\it infinitely divisible} if for any $n \in \mathbb{N}$ there exists a probability distribution $F_n$ on $(0, \infty)$ such that
$$\int_{0}^{\infty} e^{-st} {\rm d}F(t) = \left(\int_{0}^{\infty} e^{-st} {\rm d}F_n(t)\right)^n.$$
A function $g:(0, \infty)\to \mathbb{R}$ is {\it completely monotonic} if all its derivative exist and satisfies $(-1)^ng^{(n)}(x) > 0$ for all $x \in (0, \infty)$ and $n \in \{0, 1, 2, \dots \}$. The following result, known as Bernstein's theorem, characterizes completely monotonic functions, see \cite[Theorem 12a]{Wi46}.

\begin{lemma}\label{Bernstein theorem}
A function $g: (0, \infty)\to (0, \infty)$ is completely monotonic if and only if $g$ admits the following representation
$$g(s) = \int_{0}^{\infty} e^{-st} {\rm d}\alpha(t),$$ where $\alpha(t)$ is bounded and non-decreasing.
\end{lemma}

A function $g:(0, \infty) \to \mathbb{R}$ is said to be a {\it Bernstein function} if $g>0$ and the derivative of $g$ is completely monotonic on $(0, \infty)$. Let $u,$ $v>0$ and $w = v+1/v$. A function $g: (0, \infty) \to (0, \infty)$ is {\it hyperbolically completely monotonic} if $w \mapsto g(uv)g(u/v)$ is completely monotone on $(2, \infty)$. The following lemma establishes the connection between completely monotonic functions with infinitely divisible distributions, for more details see \cite[p. 450]{Fe66}.

\begin{lemma}
The function $\mathcal{L}: (0, \infty) \to (0, \infty)$ is the Laplace transform of an infinitely divisible distribution if and only if the function $\varphi(s) = -\log \mathcal{L}(s)$ is a Bernstein function on $(0, \infty),$ i.e. $\varphi'$ is completely monotonic on $(0, \infty)$ and $\varphi(0^{+})=0$.
\end{lemma}

We next recall the definition of generalized gamma convolutions, a subclass of infinitely divisible distributions. A random variable $X$ supported on $(0,\infty)$ belongs to the class of generalized gamma convolutions if its Laplace transform, $\mathcal{L}(s)=\mathbb{E}[\exp(-sX)]$, admits the following representation
$$\mathcal{L}(s)=\exp\left(-as-\int_0^\infty \ln\left(1+\frac{s}{t}\right) {\rm d}\mu(t)\right),$$
where $s\geq 0$, $a\geq 0$, and $\mu$ is a nonnegative measure on $(0,\infty)$ satisfying the appropriate integrability conditions, see \cite[p. 29]{Bo92}. The following lemma establishes the connection between hyperbolically completely monotonic functions with generalized gamma convolutions, see \cite[p. 90]{Bo92}.

\begin{lemma}\label{hcm-ggc}
A function $\mathcal{L}$ on $(0, \infty)$ is the Laplace transform of a generalized gamma convolution if and only if $\mathcal{L}(0^{+}) = 1$ and $\mathcal{L}$ is a hyperbolically completely monotonic function.
\end{lemma}

We conclude this section by recalling the definitions of Stieltjes functions, complete Bernstein functions, and Thorin–Bernstein functions. For further details and related results, we refer the reader to \cite{SSV12}.

A function $g : (0, \infty) \to [0, \infty)$ is said to be a {\it Stieltjes function} if it admits the following representation (see \cite[p. 16]{SSV12})
$$g(x) = a+\frac{b}{x}+\int_{0}^{\infty}\frac{1}{x+t} {\rm d}\mu(t), \quad \mbox{with} \quad \int_{0}^{\infty}\frac{1}{1+t} {\rm d}\mu(t)<\infty, \quad a, b \geq 0.$$
Moreover, a function $g: (0, \infty) \to [0, \infty)$ is said to be a {\it complete Bernstein function} if $g$ admits the following representation (see \cite[p. 69]{SSV12})
\begin{equation}\label{cbf definition}
g(x) = a+bx+\int_{0}^{\infty}(1-e^{-xt})m(t) {\rm d}t,\quad \mbox{with} \quad \int_{0}^{\infty}\min\{1, t\}m(t) {\rm d}t<\infty, \quad a, b \geq 0,
\end{equation}
and $m$ is a completely monotonic function on $(0, \infty)$. The following lemma relates complete Bernstein functions to Stieltjes functions, see \cite[Theorem 6.2]{SSV12} and \cite[Theorem 7.3]{SSV12} for more details.

\begin{lemma}\label{cbf-sf}
Let $g: (0,\infty) \to (0, \infty)$ be a non-trivial function. Then the following statements are equivalent.
\begin{enumerate}
\item[\bf a.] $t \mapsto g(t)$ is a complete Bernstein function on $(0,\infty)$.
\item[\bf b.] $t \mapsto {g(t)}/t$ is a Stieltjes function on $(0,\infty)$.
\item[\bf c.] $t \mapsto 1/g(t)$ is a Stieltjes function on $(0,\infty)$.
\end{enumerate}
\end{lemma}

Finally, a function $g: (0, \infty) \to [0, \infty)$ is said to be a {\it Thorin–Bernstein function} if $g$ admits the representation \eqref{cbf definition} with the additional assumption that the function $t\mapsto t\cdot m(t)$ is completely monotonic on $(0,\infty)$. Since the complete monotonicity of $t\mapsto t\cdot m(t)$ implies that of $t\mapsto m(t)$, every Thorin–Bernstein function is a complete Bernstein function.

\section{\bf NBU and NWU property of the $K$-distribution}
\setcounter{equation}{0}

In the following theorem, by assuming $a\geq b$ and $b\leq 1$, we prove that the probability density function $f_{a,b,\alpha}(x)$ of the $K$-distribution is log-convex on $(0,\infty)$. Consequently, in view of Lemma \ref{concave_theorem_f} and \ref{NBU_theorem_S}, the $K$-distribution possesses the NWU property.

\begin{theorem}\label{logconvex of pdf}
If $a\geq b,$ $b\leq 1$ and $\alpha>0,$ then the probability density function of the $K$-distribution $x \mapsto f_{a,b,\alpha}(x)$ is log-convex on $(0,\infty)$. Consequently, the corresponding survival function is log-convex on $(0,\infty)$, and hence the $K$-distribution possesses the NWU property.
\end{theorem}

\begin{proof}[\bf Proof of Theorem \ref{logconvex of pdf}]
Observe that
\begin{equation*}
\frac{f'_{a, b, \alpha}(x)}{f_{a, b, \alpha}(x)} = \left(\frac{a+b}{2}-1\right)\frac{1}{x}+\sqrt{\frac{ab}{\alpha x}}\frac{K'_{a-b}\left(2\sqrt{abx/\alpha}\right)}{K_{a-b}\left(2\sqrt{abx/\alpha}\right)}.
\end{equation*}
By using the recurrence relation (see \cite[Eq. 10.29.2]{NIST10})
\begin{equation}\label{recurrence Bessel}
K'_{\mu}(x) = -K_{\mu-1}(x)-\frac{\mu}{x}K_{\mu}(x),
\end{equation}
we arrive at
\begin{align*}
\frac{f'_{a, b, \alpha}(x)}{f_{a, b, \alpha}(x)} &= \frac{b-1}{x}-\sqrt{\frac{ab}{\alpha x}}\frac{K_{a-b-1}\left(2\sqrt{abx/\alpha}\right)}{K_{a-b}\left(2\sqrt{abx/\alpha}\right)}\\
&=\frac{b-1}{x}- 2\left(\frac{ab}{\alpha}\right)\left(4abx/\alpha\right)^{-1/2}\frac{K_{a-b-1}\left(\sqrt{4abx/\alpha}\right)}{K_{a-b}\left(\sqrt{4abx/\alpha}\right)}.
\end{align*}
Note that the ratio of the modified Bessel functions of the second kind has the following Stieltjes transform representation (see \cite[Eq. 1.3]{Is77})
\begin{equation}
x^{-1/2}\frac{K_{\mu-1}(\sqrt{x})}{K_{\mu}(\sqrt{x})} = \frac{2}{\pi^2}\int_{0}^{\infty}\frac{t^{-1}}{x+t}[J^2_{\mu}(\sqrt{t})+Y^2_{\mu}(\sqrt{t})]^{-1} {\rm d}t,\quad x>0, \quad \mu\geq 0,
\end{equation}
where $J_{\mu}$ and $Y_{\mu}$ are the Bessel functions of first and second kind of order $\mu$, respectively. By using this representation, when $a \geq b$, we deduce that
\begin{equation*}
\frac{f'_{a, b, \alpha}(x)}{f_{a, b, \alpha}(x)} = \frac{b-1}{x}-\frac{4ab}{\pi^2}\int_{0}^{\infty}\frac{t^{-1}}{4abx+\alpha t}[J^2_{a-b}(\sqrt{t})+Y^2_{a-b}(\sqrt{t})]^{-1} {\rm d}t.
\end{equation*}
Consequently, we have that
\begin{equation*}
\left(\frac{f'_{a, b, \alpha}(x)}{f_{a, b, \alpha}(x)}\right)' = \frac{1-b}{x^2}+\frac{(4ab)^2}{\pi^2}\int_{0}^{\infty}\frac{t^{-1}}{(4abx+\alpha t)^2}[J^2_{a-b}(\sqrt{t})+Y^2_{a-b}(\sqrt{t})]^{-1} {\rm d}t>0.
\end{equation*}
Hence, $x \mapsto f_{a, b, \alpha}(x)$ is log-convex on $(0, \infty)$. By using the relation (see \cite[eq. 10.25.3]{NIST10})
\begin{equation}\label{asymptotic of kmu at infinity}
K_{\mu}(x) \sim \sqrt{\pi/(2x)}e^{-x} \quad \mbox{as} \quad x \to \infty,
\end{equation}
we can verify that $\lim_{x\to \infty} f_{a, b, \alpha}(x)=0$. Consequently, in view of Lemma \ref{concave_theorem_f}, we conclude that the survival function of the $K$-distribution is also log-convex. Now, in view of Lemma \ref{NBU_theorem_S}, we obtain that the $K$-distribution (gamma-gamma distribution) has the NWU property.
\end{proof}

Let $X$ be a random variable supported on $(0,\infty)$. The following remark gives the probability density function of the random variable $Y_t= \{X-t \mid X>t\} $, where $t>0$. This result will be useful in the proof of the forthcoming theorem.

\begin{remark}\label{remark_conditional_rv}
{\em Let $X$ be a random variable supported on $(0,\infty)$, with probability density function $f$, cumulative distribution function $F$ and survival function $S$. Let $Y_t = \{ X-t \mid X>t \}$, for all $t>0$. The cumulative distribution function of the random variable $Y$ can be derived as follows:
$${\rm P}(Y_t<x) = \frac{{\rm P}(X-t<x, X>t)}{{\rm P}(X>t)}= \frac{{\rm P}(t<X<x+t)}{{\rm P}(X>t)}=\frac{F(x+t)-F(t)}{S(t)}.$$
Consequently, the probability density function of the random variable $Y_t$ is $f(x+t)/S(t)$ which can be obtained by differentiating $(F(x+t)-F(t))/S(t)$ with respect to $x$.}
\end{remark}

By using Lemma \ref{st_NWU}, the following theorem establishes that the $K$-distribution possesses the NWU property without assuming $a\geq b$. It is worth mentioning that this result does not rely on log-concavity or log-convexity arguments.

\begin{theorem}\label{a<1_NWU}
If $\min\{a, b\}\leq 1$ and $\alpha >0,$ then the $K$-distribution possesses the NWU property.
\end{theorem}

\begin{proof}[\bf Proof of Theorem \ref{a<1_NWU}]
First, we consider the case $a \leq 1$. Let $X$ follows the $K$-distribution with probability density function $f_{a, b, \alpha}$ and survival function $S_{a, b, \alpha}$. In view of Remark \ref{remark_conditional_rv}, the probability density function of the random variable $Y_t = \{X-t|X>t\}$ is $x\mapsto f_{a, b, \alpha}(x+t)/S_{a, b, \alpha}(t)$, where $t>0$. In view of Lemma \ref{st_NWU} and implication \eqref{lr_st}, in order to prove that the $K$-distribution has the NWU property it suffices to show that 
$$x \mapsto G(x) = \frac{f_{a, b, \alpha}(x+t)}{f_{a, b, \alpha}(x)S_{a, b, \alpha}(t)}$$
is increasing on $(0, \infty)$. We can verify that
$$G(x) = \frac{1}{S_{a,b,\alpha}(t)}\left(\frac{x+t}{x}\right)^{\frac{a+b}{2}-1}
\frac{K_{a-b}\left(2\sqrt{\frac{ab}{\alpha}(x+t)}\right)}{K_{a-b}\left(2\sqrt{\frac{ab}{\alpha}x}\right)}.$$ 
By using the following integral representation (see \cite[eq. 10.32.10]{NIST10})
\begin{equation}\label{integral representation for kmu}
K_{\nu}(2\sqrt{z}) = \frac{1}{2}z^{\frac{\nu}{2}}\int_{0}^{\infty}s^{-\nu-1}e^{-s-{\frac{z}{s}}} {\rm d}s,
\end{equation}
we arrive at 
$$G(x) = \frac{1}{S_{a, b, \alpha}(t)}\left(\frac{x+t}{x}\right)^{a-1}\frac{\displaystyle\int_{0}^{\infty}s^{b-a-1}e^{-s-\frac{ab}{\alpha s}(x+t)}{\rm d}s}{\displaystyle\int_{0}^{\infty}s^{b-a-1}e^{-s-\frac{ab}{\alpha s}x}{\rm d}s}.$$ 
Now, logarithmic differentiation leads to
$$\frac{G'(x)}{G(x)} = \frac{(1-a)t}{x(x+t)}+\frac{ab}{\alpha}\left[\frac{\displaystyle\int_{0}^{\infty}s^{b-a-2}e^{-s-\frac{ab}{\alpha s}x}{\rm d}s}{\displaystyle\int_{0}^{\infty}s^{b-a-1}e^{-s-\frac{ab}{\alpha s}x}{\rm d}s}-\frac{\displaystyle\int_{0}^{\infty}s^{b-a-2}e^{-s-\frac{ab}{\alpha s}(x+t)}{\rm d}s}{\displaystyle\int_{0}^{\infty}s^{b-a-1}e^{-s-\frac{ab}{\alpha s}(x+t)}{\rm d}s}\right].$$
and for $a\leq 1$, in order to show $G'(x)>0$ it is sufficient to prove the function 
$$x \mapsto r(x) = \frac{\displaystyle\int_{0}^{\infty}s^{b-a-2}e^{-s-\frac{ab}{\alpha s}x}{\rm d}s}{\displaystyle\int_{0}^{\infty}s^{b-a-1}e^{-s-\frac{ab}{\alpha s}x}{\rm d}s}$$ is decreasing on $(0, \infty)$. Note that the sign of derivative of $r(x)$ is same as the sign of the expression
\begin{equation}
\frac{ab}{\alpha}\left[\left(\int_{0}^{\infty}s^{b-a-2}e^{-s-\frac{ab}{\alpha s}x}{\rm d}s\right)^2-\int_{0}^{\infty}s^{b-a-3}e^{-s-\frac{ab}{\alpha s}x}{\rm d}s \cdot \int_{0}^{\infty}s^{b-a-1}e^{-s-\frac{ab}{\alpha s}x}{\rm d}s\right].
\end{equation}
But, by using the Cauchy-Schwarz inequality for integrals, we clearly have that 
\begin{align*}
\left(\int_{0}^{\infty}s^{b-a-2}e^{-s-\frac{ab}{\alpha s}x}{\rm d}s\right)^2 &= \left(\int_{0}^{\infty}s^{\frac{a-b-3}{2}}e^{-\frac{s}{2}}e^{-\frac{ab}{2\alpha s}x}\cdot s^{\frac{a-b-1}{2}}e^{-\frac{s}{2}}e^{-\frac{ab}{2\alpha s}x}{\rm d}s\right)^2\\
&\leq \int_{0}^{\infty}s^{b-a-3}e^{-s-\frac{ab}{\alpha s}x}{\rm d}s \cdot \int_{0}^{\infty}s^{b-a-1}e^{-s-\frac{ab}{\alpha s}x}{\rm d}s, 
\end{align*}
which implies that $x \mapsto r(x)$ is decreasing on $(0, \infty)$. Consequently, for $a \leq 1$ we have that $x \mapsto G(x)$ is increasing on $(0, \infty)$. Thus, we arrive at the desired result.

Now, consider the case when $b \leq 1$ and $b < a$. Here, by using the relation $K_{-\nu}(x) = K_{\nu}(x)$ (see \cite[eq. 10.27.3]{NIST10}), we can verify that $f_{a, b, \alpha}(x) = f_{b, a, \alpha}(x)$ and $S_{a, b, \alpha}(x) = S_{b, a, \alpha}(x)$. Thus, the proof for the case $b\leq 1$ follows immediately by symmetry.
\end{proof}

It follows from Theorem \ref{a<1_NWU} that the $K$-distribution possesses the NWU property whenever $a\leq1$ or $b\leq1$. Thus, the remaining parameter region is $a,$ $b>1$. For this parameter region, we prove in the forthcoming theorem that the $K$-distribution possesses neither the NBU nor the NWU property. The following remark plays an important role in establishing this result.

\begin{remark}\label{mean of gamma gamma}
{\em Let $X$ be the random variable which follows the $K$-distribution. Then $$\mathbb{E}[X] = \frac{2}{\Gamma(a)\Gamma(b)}\left(\frac{ab}{\alpha}\right)^{\frac{a+b}{2}}\displaystyle\int_{0}^{\infty}x^{\frac{a+b}{2}}K_{a-b}\left(2\sqrt{\frac{ab}{\alpha}x}\right){\rm d}x.$$ By using the transformation $u = 2\sqrt{{abx}/{\alpha}}$, we arrive at
$$\mathbb{E}[X] = \frac{\alpha}{2^{a+b}\Gamma(a+1)\Gamma(b+1)}\int_{0}^{\infty}u^{a+b+1}K_{a-b}(u){\rm d}u.$$
By using the relation \cite[eq. 10.43.19]{NIST10}
$$\int_{0}^{\infty}t^{\mu-1}K_{\nu}(t){\rm d}t = 2^{\nu-2}\Gamma\left(\frac{1}{2}\mu-\frac{1}{2}\nu\right)\Gamma\left(\frac{1}{2}\mu+\frac{1}{2}\nu\right),$$ we conclude that $\mathbb{E}[X] = \alpha.$}
\end{remark}

In the following theorem, we prove that for $a,$ $b>1$ the $K$-distribution possesses neither the NWU nor the NBU property. The proof proceeds by contradiction and relies on the fact that if $X\leq_{st}Y$, then $\mathbb{E}[X]\leq\mathbb{E}[Y]$ (see Remark \ref{remark on expectation}).

\begin{theorem}\label{a>1_NWU}
If $a,$ $b > 1$ and $\alpha >0,$ then the $K$-distribution neither possesses the NWU property nor the NBU property.
\end{theorem}

\begin{proof}[\bf Proof of Theorem \ref{a>1_NWU}]
Let $X$ be a random variable following the $K$-distribution and $Y_t=\{X-t\mid X>t\}$. In view of Lemma \ref{st_NWU}, to prove that $X$ possesses neither the NWU nor the NBU property, it suffices to show that neither of the stochastic orders $X\leq_{st}Y_t$ nor $Y_t\leq_{st}X$ hold. In view of Remark \ref{remark on expectation}, we show that there exist values of $t$ for which $\mathbb{E}[X]<\mathbb{E}[Y_t]$, as well as values of $t$ for which $\mathbb{E}[X]>\mathbb{E}[Y_t]$ hold. In view of \cite[p. 34]{UNS13}, we can write
\begin{equation}\label{definition of m}
m(t)=\mathbb{E}[Y_t]=\frac{1}{S_{a, b, \alpha}(t)}\int_{t}^{\infty}S_{a, b, \alpha}(u){\rm d}u.
\end{equation}
Note that $m(0) = \displaystyle\int_{0}^{\infty}S_{a, b, \alpha}(u){\rm d}u = \mathbb{E}[X]$ and $\mathbb{E}[X] = \alpha$ (see Remark \ref{mean of gamma gamma}). Moreover, we can verify that $$m'(t) = -1+\frac{f_{a, b, \alpha}(t)}{S_{a, b, \alpha}(t)}m(t).$$ In view of the asymptotic relations \cite[eq. 10.30.2 \& eq. 10.30.3]{NIST10}
\begin{equation}\label{asymptotic kmu at zero}
K_{\nu}(z) \sim
\begin{cases}
\displaystyle \frac{1}{2}\Gamma(\nu)\left(\frac{z}{2}\right)^{-\nu},
& \nu>0,\\[8pt]
-\ln z, & \nu=0,
\end{cases}
\qquad \mbox{as } z\to 0,
\end{equation}
 when $a>1$, we have $\lim_{t \to 0}f_{a, b, \alpha}(t) = 0$. Consequently, $\lim_{t \to 0}m'(t) = -1$. Therefore, for sufficiently small $\epsilon>0$, we have $m(\epsilon)<m(0)$. Since $m(t)=\mathbb{E}[Y_t]$ and $m(0)=\mathbb{E}[X]$, this is equivalent to $\mathbb{E}[Y_{\epsilon}]<\mathbb{E}[X]$.

On the other hand, by repeated applications of Bernouli-Hospital's rule, we arrive at $$\lim_{t\to\infty}m(t)=\lim_{t\to\infty}-\frac{f(t)}{f'(t)}.$$ Thus, by using the asymptotic relation \eqref{asymptotic of kmu at infinity} and the recurrence relation \eqref{recurrence Bessel}, we have that $\lim_{t\to\infty}m(t)=\infty$. Consequently, there exists a sufficiently large $\epsilon_1>0$ such that $m(\epsilon_1)>m(0)$. That is, $\mathbb{E}[Y_{\epsilon_1}]<\mathbb{E}[X]$.
\end{proof}

The next alternative proof of Theorems \ref{a<1_NWU} and \ref{a>1_NWU} establishes a stronger result than those stated in the theorems. Specifically, we provide a complete characterization of the NBU and NWU properties of the $K$-distribution.

\begin{proof}[\bf Alternative proof of Theorems \ref{a<1_NWU} and \ref{a>1_NWU}]
This alternative proof is based on the fact that the $K$-distribution is derived as a product of two independent gamma distribution. Let $U$ and $V$ be two independent random variables such that $U$ follows a gamma distribution with shape parameter $a$ and scale
parameter $\lambda_1$, while $V$ follows a gamma distribution with shape parameter $b$ and scale parameter $\lambda_2$. Their respective probability density functions are given by
\begin{equation*}
f_{U}(x) = \frac{\lambda_1^a}{\Gamma(a)}x^{a-1}e^{-\lambda_1 x}\quad \mbox{and} \quad f_{V}(x) = \frac{\lambda_2^b}{\Gamma(b)}x^{b-1}e^{-\lambda_2 x},
\end{equation*}
where $a,$ $b,$ $\lambda_1,$ $\lambda_2,$ $x>0$. If we consider the product $X = UV$, then with $c = \lambda_1  \lambda_2$ the product density formula (see for example \cite[eq. 4]{ST66}) and the standard integral representation of the modified Bessel function of the second kind \eqref{integral representation for kmu} imply that
\begin{equation*}
f_{X}(x) = \frac{2c^{\frac{a+b}{2}}}{\Gamma(a)\Gamma(b)}x^{\frac{a+b}{2}-1}K_{a-b}(2\sqrt{cx}),
\end{equation*}
which is exactly the density function of $K$-distribution for $c={ab}/{\alpha}$. Thus, we have $f_{X}(x) = f_{a, b, \alpha}(x)$. Now, consider the case when $a\in (0, 1]$. The gamma density of the random variable $U$ can be written as
$$f_{U}(x) =\frac{\lambda_1^a}{\Gamma(a)} x^{-(1-a)} e^{-\lambda_1 x}.$$
Note that in the above expression both factors containing $x$ are completely monotonic on $(0, \infty)$ and so is their product. It follows that the survival function $$S_{U}(x)=\int_{x}^{\infty} f_{U}(t) {\rm d}t$$ is also completely monotone on $(0, \infty)$ for all $\lambda_1>0$ and $a \in (0, 1]$. Thus, by using Bernstein's theorem (see Lemma \ref{Bernstein theorem}), we have that
\begin{equation*}
    S_{U}(x) = \int_{0}^{\infty}e^{-xs} {\rm d}\mu(s),
\end{equation*}
for a probability measure $\mu$. Now, observe that the survival function $S_{a, b, \alpha}$ of the positive random variable $X$ can be written as $$S_{a, b, \alpha}(x) = \mathbb{E}[S_{U}\left(X/Y\right)].$$ Indeed, for $U$ and $V$ independent and positive random variables, $$S_{a, b, \alpha}(x) = P(X>x) = P(UV >x).$$ Using the law of total probability (conditional expectation) we obtain that $$P(U V >x) =\mathbb{E}\left[P(U V>x\mid V)\right]$$ and consequently, we arrive at
\begin{equation}\label{rep_S_X}
S_{a, b, \alpha}(x) = \mathbb{E}\left[P(U  V>x\mid V)\right]=\int_{0}^{\infty}S_{U}\left({x}/{v}\right)f_{V}(v){\rm d}v,
\end{equation}
where $f_{V}$ is the density function of the random variable $V$. For every fixed $v>0$ and $a \in (0, 1]$ the function $x \mapsto S_{U}\left({x}/{v}\right)$ is completely monotonic and it is known that positive mixtures preserve the complete monotonicity. Consequently, in view of \eqref{rep_S_X}, $x \mapsto S_{a, b, \alpha}(x)$ is completely monotone on $(0, \infty)$. Thus, by using Bernstein's theorem (see Lemma \ref{Bernstein theorem}), we have that
\begin{equation*}
    S_{a, b, \alpha}(x) = \int_{0}^{\infty}e^{-xs}{\rm d}\nu(s),
\end{equation*}
where $\nu$ is a probability measure on $[0, \infty)$. This implies that
$$f_{a, b, \alpha}(x) = -S'_{a, b, \alpha}(x) = \int_{0}^{\infty} s e^{-xs}{\rm d}\nu(s)$$ and the hazard rate function is
\begin{equation*}
h_{X}(x) = \frac{f_{a, b, \alpha}(x)}{S_{a, b, \alpha}(x)} = \frac{\displaystyle\int_{0}^{\infty}se^{-xs}{\rm d}\nu(s)}{\displaystyle\int_{0}^{\infty}e^{-xs}{\rm d}\nu(x)}.
\end{equation*}
By differentiating $h_{X}(x)$, we arrive at
\begin{align}\label{eq_der_h_X}
h'_{X}(x) = -\frac{\displaystyle\int_{0}^{\infty}s^2e^{-xs}{\rm d}\nu(s)}{\displaystyle\int_{0}^{\infty}e^{-xs}{\rm d}\nu(s)}+\left(\frac{\displaystyle\int_{0}^{\infty}se^{-xs}{\rm d}\nu(s)}{\displaystyle\int_{0}^{\infty}e^{-xs}{\rm d}\nu(s)}\right)^2.
\end{align}
Now, let us define the exponentially tilted measure
\begin{equation}\label{tilted}
{\rm d}\nu_{x}(s) = \frac{e^{-xs}}{\displaystyle\int_{0}^{\infty}e^{-xu}{\rm d}\nu(u)}{\rm d}\nu(s).
\end{equation}
Note that $\displaystyle\int_{0}^{\infty}{\rm d}\nu_{x}(s)=1$.
By using the exponentially tilted measure \eqref{tilted} in \eqref{eq_der_h_X}, we obtain that
$$h'_{X}(x) = -\mathbb{E}_{\nu_x}[s^2]+(\mathbb{E}_{\nu_x}[s])^2 = -\text{Var}_{\nu_x}(s).$$ Since the variance is non negative, we conclude that $h'_{X}(x)<0$ \footnote{We can also obtain this result by applying the Cauchy–Schwarz inequality for integrals.}. Thus, the random variable satisfies the decreasing failure rate property (DFR), that is, its failure rate or hazard rate is decreasing. Therefore, $X$ is NWU (see Lemma \ref{dfr-nwu}). By symmetry, we can clearly interchange the parameters $a$ and $b$. Thus, we conclude that $\min\{a, b\}\leq 1$ implies $X$ is DFR and then NWU.

Now, we are going to show that if both shapes $a$ and $b$ exceed one, then NWU fails near the origin. Let $m = \min \{a, b\}$. Assume that $m > 1$. For the case $a \neq b$, by using the relation \eqref{asymptotic kmu at zero}, we can verify that as $x \to 0$ the gamma-gamma density satisfies $f_{a, b, \alpha}(x) \sim C_1 x^{m-1}$ and then the cumulative distribution function $F_{a, b, \alpha}(x) \sim C_2 x^{m}$ for some positive constants $C_1$ and $C_2$. Now, observe that $$H_{X}(x) = -\log(S_{a, b, \alpha}(x))=-\log(1-F_{a, b, \alpha}(x)) \sim F_{a, b, \alpha}(x)$$ as $x \to 0$ and therefore $$H_{X}(2x)/H_{X}(x) \to 2^{m}>2\quad \mbox{as}\quad x \to 0.$$ But the NWU property is equivalent to the subadditivity of $H_{X}(x)$ (see eq. \eqref{nbu_equation}), that is $$H_{X}(x+y) \leq H_{X}(x) + H_{X}(y),$$ and we should have $H_{X}(2x)\leq 2H_{X}(x)$. Thus, we arrive to a contradiction and the NWU property fails in this case. For the case $a=b=m$ by using the relation \eqref{asymptotic kmu at zero}, we arrive at $F_{a, b, \alpha}(x) \sim C_3x^{m}\log(C_4/x)$ as $x \to 0$, where $C_3$ and $C_4$ are positive constants. Consequently, we have that $$H_{X}(2x)/H_{X}(x) \to 2^{m}>2\quad \mbox{as} \quad x \to 0,$$ so the NWU property is fails for the case $a=b$ also. Thus, we have shown the exact characterization: $X$ is NWU if and only if $\min\{a, b\}\leq 1$.

Now, we are going to show that NBU fails for every finite positive parameter pair $a$ and $b$. For fixed order $\nu$, as $z \to \infty$, we have (see \cite[eq. 10.40.2]{NIST10})
$$K_{\nu}(z) \sim \sqrt{\frac{\pi}{2z}}e^{-z} \left(1+\frac{4\nu^2-1}{8z}+\mathcal{O}(z^{-2})\right).$$
Consequently, the gamma-gamma density has the following asymptotic behaviour
$$f_{a, b, \alpha}(x) = \frac{\sqrt{\pi}c^{\frac{a+b}{2}-\frac{1}{4}}}{\Gamma(a)\Gamma(b)} x^{\frac{a+b}{2}-\frac{5}{4}} e^{-2\sqrt{cx}} \left(1+\frac{4(a-b)^2-1}{16\sqrt{cx}}+\mathcal{O}(x^{-1})\right)$$ as $x \to \infty$, and the main feature is the factor $e^{-2\sqrt{cx}}$, which decays more slowly than an ordinary exponential $e^{-cx}$. By integrating the leading tail, we arrive at
$$S_{a, b, \alpha}(x) = \int_{x}^{\infty} f_{a, b, \alpha}(t){\rm d}t \sim \frac{\sqrt{\pi}c^{\frac{a+b}{2}-\frac{3}{4}}}{\Gamma(a)\Gamma(b)} x^{\frac{a+b}{2}-\frac{3}{4}} e^{-2\sqrt{cx}}, \quad x \to \infty$$
and therefore
\begin{equation}\label{limit of survival}
-\log S_{a, b, \alpha}(x) = 2\sqrt{cx}-\left(\frac{a+b}{2}-\frac{3}{4}\right)\log(x)+\mathcal{O}(1)
\end{equation}
and in particular $H_{X}(x) \sim 2\sqrt{cx}$ as $x \to \infty$. Consequently, for every integer $n\geq 2$, we arrive at
$$H_{X}(nx)/H_{X}(x) \to \sqrt{n}\quad \mbox{as}\quad x \to \infty.$$ On the other hand it is known that the NBU property is equivalent to the superadditivity of $H_{X}(x) = -\log(S_{a, b, \alpha}(x))$ (see eq. \eqref{nbu_equation}), that is
$$H_{X}(x+y) \geq H_{X}(x) + H_{X}(y).$$ If $X$ were NBU, superadditivity would imply $H_{X}(nx) \geq n H_{X}(x)$, which would contradict $\sqrt{n}<n$, and therefore $X$ is never NBU for finite $a,$ $b>0$.
\end{proof}

\begin{remark}
{\em A compact way to understand the preceding classification is to consider $A(t)={H_{X}(t)}/{t},$ where $H_{X}(t) = -\log(S_{a, b, \alpha}(t))$. The function $A(t)$ is called in the literature as the normalized cumulative hazard function. If $X$ is NBU, then $H_{X}(t)$ is superadditive and hence for every integer $n \geq 1$, we have $H_{X}(nt) \geq nH_{X}(t)$ so that $A(nt) \geq A(t)$. Similarly, if $X$ is NWU, then $H_{X}$ is subadditive and $H_{X}(nt) \leq H_{X}(t)$ so that $A(nt) \leq A(t)$. These scaling inequalities are weaker than the global monotonicity of $A$, the later corresponds to the stronger IFRA/DFRA (increasing failure rate in average/ decreasing failure rate in average) ageing classes. Recall that for the gamma-gamma distribution as $t \to \infty$, we have
$$H_{X}(t) = 2\sqrt{ct}+ \mathcal{O}(\log(t))\quad  \mbox{and} \quad A(t) \sim 2\sqrt{c}/\sqrt{t} \to 0.$$ If $X$ is NBU, then for any fixed $t >0$, $A(nt) \geq A(t)>0$ for all $n$, contradicting $A(nt) \to 0$. Thus, the gamma-gamma distribution cannot be NBU for any positive $a, b$.}

{\em Let $m = \min \{a, b\}>1$. If $a \neq b$, then $H_{X}(t)\sim C t^m$ as $t \to 0$, whereas for the case $a=b=m$, we have
$$H_{X}(t)\sim C t^m \log(1/t)\quad \mbox{as}\quad t \to 0.$$ In either cases
$$A(t) = {H_{X}(t)}/{t} \to 0\quad \mbox{as}\quad t \to 0.$$ If $X$ is NWU, then
$$H_{X}(t)=H_{X}(n\cdot t/n)\leq n\cdot H_{X}(t/n),$$ and therefore $A(t)\leq A(t/n)$. Now let $n \to \infty$, and we arrive at $A(t) \leq 0$, impossible because $H_{X}(t)>0$ for all $t>0$. Hence, when $a$, $b>1$ the gamma-gamma distribution is not NWU. This yields the same complete picture in a particular concise form:
\begin{align*}
    \min\{a, b\}\leq 1 &\implies X \text{is} \text{DFR} \implies X \text{is} \text{NWU},\\
    \text{while}  a>1, b>1 &\implies X \text{is not NWU},\\
    \text{and}  a, b>0 &\implies X \text{is not NBU}.
\end{align*} }
\end{remark}

\begin{remark}\label{nwuProofGeneral}
{\em Another way to obtain the above classification is a direct approach of survival inequalities. For this, observe that if $U>0$ is NWU and $V>0$ is independent of $U$, then $X = U V$ is also NWU. Indeed, since $U$ is NWU, we have
$$S_{U}\left(\frac{x+y}{v}\right) \geq S_{U}\left(\frac{x}{v}\right) S_{U}\left(\frac{y}{v}\right)$$
for every $x,$ $y,$ $v>0$. On the other hand $$S_{X}(x+y) = \mathbb{E}\left[S_{U}\left(\frac{x+y}{V}\right)\right]$$ and therefore
$$S_{X}(x+y) \geq \mathbb{E}\left[S_{U}\left(\frac{x}{V}\right)S_{U}\left(\frac{y}{V}\right)\right].$$ Observe that the function $v \mapsto S_{U}\left({x}/{v}\right)$ and $v \mapsto S_{U}\left({y}/{v}\right)$ are increasing, so by the Chebyshev covariance inequality, we clearly have
$$\mathbb{E}\left[S_{U}\left(\frac{x}{V}\right)S_{U}\left(\frac{y}{V}\right)\right]\geq \mathbb{E}\left[S_{U}\left(\frac{x}{V}\right)\right]\mathbb{E}\left[S_{U}\left(\frac{y}{V}\right)\right],$$ that is $$S_{X}(x+y) \geq S_{X}(x)S_{X}(y).$$

Now, let $U$ and $V$ be two independent random variables such that $U$ follows a gamma distribution with shape parameter $a$ and scale parameter $\lambda_1$, while $V$ follows a gamma distribution with shape parameter $b$ and scale parameter $\lambda_2$. We know that a gamma distribution with shape parameter at most one is DFR, hence NWU. Therefore, if $\min\{a, b\} \leq 1$, at least one factor is NWU, and by the above closure argument we derive that $X = UV$ is NWU. Conversely, suppose that $a>1$, and $b>1$, and let $m = \min \{a, b\}>1$. If $X$ were NWU, then by taking $x = y =t$ we would have $S_{a, b, \alpha}(2t) \geq \left(S_{a, b, \alpha}(t)\right)^2$, which would imply $$F_{a, b, \alpha}(2t)/F_{a, b, \alpha}(t) \leq 2-F_{a, b, \alpha}(t).$$ Therefore any NWU distribution with the cdf $F_{a, b, \alpha}(x) \to 0$ as $t \to 0$ must satisfy $$\limsup_{t \to 0}{F_{a, b, \alpha}(2t)}/{F_{a, b, \alpha}(t)} \leq 2.$$ For the gamma-gamma distribution $F_{a, b, \alpha}(t) \sim C t^{m}$ for $a \neq b$ and $F_{a, b, \alpha}(t) \sim C t^{m} \log(1/t)$ for $a=b=m$, as $t \to 0$. So in both cases $${F_{a, b, \alpha}(2t)}/{F_{a, b, \alpha}(t)} \to 2^{m}>2,$$ which contradicts NWU. }
\end{remark}

The beta prime family provides a natural ratio analogue of the gamma-gamma distribution. Let $Y$ be a random variable which follows the beta prime distribution with parameters $a>0$ and $b>0$, that is its density function is
$$f_{Y}(x) = \frac{1}{\beta(p, q)}\frac{x^{a-1}}{(1+x)^{a+b}}, \quad x>0.$$
Equivalently, if $U$ and $V$ be two independent random variables such that $U$ follows a gamma distribution with shape parameter $a$ and scale parameter $\lambda_1$, while $V$ follows a gamma distribution with shape parameter $b$ and scale parameter $\lambda_2$, then the ratio $Y = U/V$ has the beta-prime distribution. By using the similar approach as before we are going to show the following characterization.

\begin{theorem}\label{neither nbu nor nwu}
The beta prime distribution is NWU if and only if $a \leq 1$ and the beta prime distribution is never NBU for $a,$ $b>0$, hence when $a>1$, the beta prime distribution is neither NWU nor NBU.
\end{theorem}

\begin{proof}[\bf Proof of Theorem \ref{neither nbu nor nwu}]
We first show the NWU part. Suppose that $a \in (0, 1]$. The density is completely monotonic on $(0, \infty)$ as it is the product of the completely monotonic function $x \mapsto x^{a-1}$ and $x \mapsto (1+x)^{-(a+b)}$. Hence, the survival function $S_{Y}$ is completely monotonic as well on $(0, \infty)$ for all $a \in (0, 1]$ and $b>0$. By Bernstein's theorem, $S_{Y}$ is a mixture of exponentials (see Lemma \ref{Bernstein theorem}) and consequently the random variable $Y$ is DFR, and therefore NWU. For the converse, we shall use the behaviour at the origin. Since, $(1+x)^{-(a+b)} \to 1$ as $x \to 0$, we have that
$$f_{Y}(x) \sim {x^{a-1}}/{\beta(a, b)}\quad \mbox{and} \quad F_{Y}(x) \sim {x^a}/{(a\beta(a,b))}$$
as $x \to 0$. If $Y$ were NWU, then we would have $S_{Y}(2x) \geq (S_{Y}(x))^2$ or equivalently $F_{Y}(2x)/F_{Y}(x)\leq 2$. But, the preceding asymptotic relation implies that $F_{Y}(2x)/F_{Y}(x) \to 2^a$, and hence NWU requires $2^a\leq 2$, that is $a\leq 1$. Thus, $Y$ is NWU if and only if $a \leq 1$. It remains to consider the NBU property. As $x \to \infty$,
$$f_{Y}(x)\sim {x^{-b-1}}/{\beta(a, b)} \quad\mbox{and} \quad S_{Y}(x)\sim {x^{-b}}/{(b\beta(a, b))}.$$ Thus, we have
$$H_{Y}(x)=-\log S_{Y}(x)=b\log(x)+\mathcal{O}(1) \quad \mbox{as}\quad  x \to \infty,$$
and in particular $H_{Y}(x)/x \to 0$ as $x \to \infty$. If $Y$ were NBU, then $H_{Y}$ would be superadditive and for every integer $n \geq 1$ we would have $H_{Y}(nx) \geq n H_{Y}(x)$. For fixed $x > 0$, this would imply $${H_{Y}(nx)}/{nx} \geq {H_{Y}(x)}/{x}>0,$$ contradicting ${H_{Y}(t)}/{t}\to 0.$ Therefore, beta prime distribution with positive parameters is not NBU.
\end{proof}

\begin{remark}
{\em Following the discussion in Remark \ref{nwuProofGeneral}, it is clear that if $X>0$ is NWU, $Y>0$ is independent of $X$, then $Z = X Y$ is also NWU. Thus, in fact one NWU factor is enough, no ageing assumption is required on the other positive independent factor. Since multiplication is commutative, the same conclusion holds if $Y$, rather than $X$, is NWU. The quotient case is immediate. More precisely, if $X>0$ is NWU and $Y>0$ is independent of $X$, then $X/Y$ is NWU. More generally, $X/Y$ is NWU whenever either $X$ is NWU or $Y^{-1}$ is NWU.

Finally, we would like to mention that the DFR property is also preserved under independent multiplication. More precisely, if $X$ is DFR and $Y>0$ is independent of $X$, then $XY$ is DFR. For this recall that a distribution is DFR if and only if its survival function is log-convex (see \cite[Theorem 4.1]{BP96}). For each fixed $v>0$, the survival function $t \mapsto S_{X}(t/v)$ is log-convex. Moreover, $S_{XY}(t) = \mathbb{E}\left[S_{X}\left(t/Y\right)\right]$ is also log-convex, since the H\"{o}lder-Rogers inequality for integrals implies that
$$S_{XY}(\theta x+(1-\theta)y) \leq \mathbb{E}\left[\left(S_{X}\left(\frac{x}{Y}\right)\right)^{\theta}\left(S_{X}\left(\frac{y}{Y}\right)\right)^{1-\theta}\right]\leq \left(S_{XY}(x)\right)^{\theta}\left(S_{XY}(y)\right)^{1-\theta}$$ for every $\theta \in [0, 1]$, $x,$ $y>0$. Thus, $XY$ is DFR. This is consistent with the general closure of DFR distribution under mixtures \cite{FE01}.}
\end{remark}

\section{\bf Infinite divisibility properties of the survival function}
\setcounter{equation}{0}

In this section we are going to present a more complete picture of the survival function of the gamma-gamma distribution. Our aim is to prove much stronger statements than the fact that the gamma-gamma distribution satisfies the decreasing failure rate (DFR) property and hence the new is worse than used property. The main theorem of this section is the following result.

\begin{theorem}\label{theorem on infinite divisibility}
    If $\min \{a, b\}\leq 1$ and $\alpha >0$, then the following statements are valid.
    \begin{enumerate}
        \item[\bf a.] $S_{a, b, \alpha}(x)$ is a hyperbolically completely monotone function.
        \item[\bf b.] $H_{X}(x) = -\log S_{a, b, \alpha}(x)$ is a Thorin-Bernstein function and hence a complete Bernstein function.
        \item[\bf c.] The hazard function $h_{X}(x) = H'_{X}(x)$ is a Stieltjes function.
        \item[\bf d.] The Mills ratio $M_{X}(x) = S_{a, b, \alpha}(x)/f_{a, b, \alpha}(x)=1/h_{X}(x)$ is a complete Bernstein function.
        \item[\bf e.] The normalized cumulative hazard $H_{X}(x)/x$ is a Stieltjes function.
    \end{enumerate}
\end{theorem}

\begin{proof}[\bf Proof of Theorem \ref{theorem on infinite divisibility}]
{\bf a.} Let $a \in (0, 1)$ and $U$ follows the gamma distribution with shape parameter $a$ and rate parameter $\lambda_1$. Then, we have that
\begin{equation}\label{survival gamma}
S_{U}(x) = \frac{\Gamma(a, \lambda_1 x)}{\Gamma(a)}=\int_{\lambda_1}^{\infty}e^{-rx}w_{a}(r){\rm d}r,\quad \mbox{where}\quad  w_{a}(r) = \frac{\lambda^a_1}{\Gamma(a)\Gamma(1-a)}\frac{(r-\lambda_1)^{-a}}{r},\ r>\lambda_1,
\end{equation}
where $$\Gamma(\mu,x)=\displaystyle\int_x^{\infty}t^{\mu-1}e^{-t}\,{\rm d}t$$ denotes the upper incomplete gamma function. Indeed since (see \cite[eq. 5.9.1]{NIST10})
$$x^{a-1} = \frac{1}{\Gamma(1-a)}\int_{0}^{\infty}e^{-tx}t^{-a}{\rm d}t,$$ and consequently,
$$\Gamma(a, \lambda_1 x) = \frac{e^{-\lambda_1 x}}{\Gamma(1-a)}\int_{0}^{\infty}\frac{e^{-\lambda_1 x t} t^{-a}}{1+t}{\rm d}t.$$
By using the transformation $r = \lambda_1(1+t)$ in the above equation, we arrive at \eqref{survival gamma}. Let $b,$ $\lambda_2>0$ and $V$ follows the gamma distribution with shape parameter $b$ and rate parameter $\lambda_2$. By using the transformation $t = (r-\lambda_1)/\lambda_1$, we obtain that 
$$\int_{\lambda_1}^{\infty} w_{a}(r){\rm d}r = 1.$$ 
Now, let $R_0$ be the random variable independent of $V$ and supported on $(\lambda_1, \infty)$ with density function $w_{a}(r)$. Let $X = U  V$. The survival function $S_{a, b, \alpha}(x)$ of $X$ is given as
$$S_{a, b, \alpha}(x)(x) = \mathbb{E}[S_{U}(x/V)] = \mathbb{E}[\exp(-xR_0/V)],$$ and thus for $R = R_0/V$, we have $S_{X}(x) = \mathbb{E}[e^{-xR}]$. Moreover, if we define the random variables $B = \lambda_1/R_0$, and by using the change of variables we obtain that $B$ follows the beta distribution with shape parameters $a$ and $1-a$. By defining $G = \lambda_2 V$, it follows that $G$ follows the gamma distribution with shape parameter $b$ and rate parameter $1$. Consequently, we have $R = (\lambda_1/B)/(G/\lambda_2)$. The density function of the random variable $B$ and $G$ are
$$f_{B}(u) = \frac{u^{a-1}(1-u)^{-a}}{\Gamma(a)\Gamma(1-a)} \quad \mbox{and}\quad f_{G}(t) = \frac{1}{\Gamma(b)}t^{b-a}e^{-t},$$ respectively. Since $B$ and $G$ are independent, we can deduce the density $f_{R}(x)$ of $R$ by conditioning on $B$. For this let $c = \lambda_1\lambda_2$. For fixed $u \in (0, 1)$, the relation $R = c/(u G)$ and the change of variable formula yields
\begin{align*}
f_{R \mid B}(x \mid u) &= f_{G}\left(\frac{c}{ux}\right)\left| \frac{{\rm d}G}{{\rm d}x}\right| = \frac{c^{b}}{\Gamma(b)}\frac{e^{-c/(ux)}}{u^bx^{b+1}}.
\end{align*}
Consequently, from the relation $$f_{R}(x) = \int_{0}^{1} f_{X\mid B}(x \mid u) f_{B}(u) {\rm d}u,$$ we arrive at
$$f_{R}(x) = \frac{c^{b}x^{-b-1}}{\Gamma(b)\Gamma(a)\Gamma(1-a)}\int_{0}^{1}u^{a-b-1}(1-u)^{-a}e^{-c/(ux)} {\rm d}u, \quad x>0.$$ By using the transformation $t = {(1-u)}/{u}$, we obtain that 
$$f_{R}(x) = \frac{c^bx^{-b-1}}{\Gamma(b)\Gamma(a)}\frac{e^{-c/x}}{\Gamma(1-a)}\int_{0}^{\infty}t^{-a}(1+t)^{b-1}e^{-ct/x} {\rm d}t.$$ By using the integral representation for the Tricomi confluent hypergeometric function (see \cite[eq.
13.4.4]{NIST10}) $$\psi(a, b, z) = \frac{1}{\Gamma(a)}\int_{0}^{\infty}e^{-zt}t^{a-1}(1+t)^{b-a-1} {\rm d}t,$$ we arrive at
\begin{equation}\label{pdf of R}
f_{R}(x) = \frac{c^b}{\Gamma(a)\Gamma(b)}x^{-b-1}e^{-c/x}\psi(1-a, 1+b-a, c/x), \quad x>0.
\end{equation}
Consequently, we have that
\begin{equation*}
S_{a, b, \alpha}(x) = \mathbb{E}\left[e^{-xR}\right]=\int_{0}^{\infty}e^{-xr}f_{R}(r) {\rm d}r.
\end{equation*}
Now, we are going to show that $R$ is a generalized gamma convolution. Indeed $R = c/(B G)$, $c = \lambda_1 \lambda_2$. Observe that $(1-B)/B$ follows the beta prime distribution of shape parameters $1-a$ and $a$. The beta prime distribution is in fact the ratio of the independent gamma variables. The ratio of two gamma distribution is a generalized gamma convolution (see \cite[Theorem 23]{BPSV26}) and hence $B^{-1} = 1+(1-B)/B$ is a generalized gamma convolution, because the class of generalized gamma convolution is closed under convolution with a nonnegative constant \cite[p. 29]{Bo92}. On the other hand $G^{-1}$ has density $$\xi_{b}(x) = \frac{1}{\Gamma(b)}x^{-b-1}e^{-1/x},$$ which is clearly hyperbolically completely monotonic. Since as a function of $w = v + 1/v$ the function $w \mapsto \xi_b(uv)\xi_{b}(u/v) = c(u) \exp(-w/u)$ is completely monotone (see \cite[p. 3]{Bo92}). Thus, $B^{-1}G^{-1}$ is a generalized gamma convolution, since the class of generalized gamma convolution is closed under product of independent positive random variable by Bondesson's multiplicative closure theorem.

Thus, we proved that the random variable $R$ is a generalized gamma convolution; combining this with the Bondesson's characterization (see Lemma \ref{hcm-ggc}), which states that the Laplace transform of a generalized gamma convolution is a hyperbolically completely monotone function, we have that $S_{a, b, \alpha}$ is indeed hyperbolically completely monotone.

{\bf b. \& c. \& d. \& e.} The Thorin representation implies that (see \cite[p. 29]{Bo92})
$$H_{X}(x) = \zeta x+\int_{0}^{\infty}\log(1+x/t) {\rm d}T(t),$$
where $\zeta\geq 0$, and ${\rm d}T(t)$ is a nonegative measure on $(0, \infty)$. Consequently, we arrive at
$$h_{X}(x) = \zeta+\int_{0}^{\infty}\frac{{\rm d}T(t)}{x+t}.$$ This is a Stieltjes function representation. In fact a gamma-gamma distribution tail satisfies $h_{X}(x) \sim \sqrt{c/x}$ as $x \to \infty$ (see \eqref{limit of survival}) and therefore the constant $\zeta$ is zero.

Finally, recall that a nonzero function is a Stieltjes function if and only if its reciprocal is a complete Bernstein function (see Lemma \ref{cbf-sf}); moreover if $g$ is complete Bernstein function, then $g(x)/x$ is a Stieltjes function (see Lemma \ref{cbf-sf}). Applying these to $h_{X}$ and $H_{X}$ proves the remaining assertions.
\end{proof}

\begin{remark}
{\em Theorem \ref{theorem on infinite divisibility} is substantially stronger than DFR. Since a Stieltjes function is completely monotone, that is $$(-1)^nh^{(n)}_{X}(x) \geq 0, \quad \mbox{for all} \quad n \in \{0, 1, 2, \cdots\}\quad \mbox{and} \quad x>0.$$ Thus, not only $h_{X}$ is decreasing, but all its derivatives alternate in sign. Likewise, because $M_{X}$ is complete Bernstein function, $M'_{X}(x) \geq 0$, $M_{X}''(x)\leq 0$ and $M_{X}'(x)$ is completely monotone. Hence, the Mills ratio of the gamma-gamma distribution in the NWU region is increasing and concave. }
\end{remark}

The following remark establishes the uniqueness of the mode of the random variable $R$, whose density function is given in \eqref{pdf of R}, and also provides simple bounds for the mode.

\begin{remark}
{\em We note that, the reciprocal rate $W = c/R$ has the simpler density
$$f_{W}(x) = \frac{x^{b-1}e^{-x}}{\Gamma(a)\Gamma(b)}\psi(1-a, 1+b-a, x), \quad x>0.$$ Moreover, it can be shown that if $0<b<a<1$, then $f_{W}$ is logarithmically completely monotone, that is $$(-1)^n \frac{{\rm d}^n}{{\rm d}x^n}\log f_{W}(x) \geq 0, \quad n \in \{1, 2, \cdots\}\quad \mbox{and} \quad x >0.$$ The proof of this affirmation is based on the fact that
\begin{align*}
-\frac{{\rm d}}{{\rm d}z}\log f_{W}(x) &= 1+\frac{1-b}{x}+(1-a)\frac{\psi(2-a, 2+b-a, x)}{\psi(1-a, 1+b-a, x)}\\
&= 1+\frac{1-b}{x}+(1-a)\int_{0}^{\infty}\frac{1}{x+t} \frac{t^{-(1+b-a)}e^{-t}|\psi(1-a, 1+b-a, te^{{\rm i}\pi})|^{-2}}{\Gamma(2-a)\Gamma(1-b)} {\rm d}t,
\end{align*}
where we used the well-known Stieltjes transform representation of the quotient of the Tricomi confluent hypergeometric function with conditions $1-a>0,$ and $1+b-a<1$ (see \cite{IK79}).

Finally, let $\phi(x) = x+(1-a)xQ(x),$ where
$$ Q(x) = \frac{\psi(2-a, 2+b-a, x)}{\psi(1-a, 1+b-a, x)}.$$ Since $Q$ is a Stieltjes function, $xQ(x)$ is a complete Bernstein function and is increasing. Therefore, $\phi$ is strictly increasing for $a<1,$ and $b<a$. A direct calculation from $f_{R}(x) = (c/x^2)f_{W}(c/x)$ implies that
$$x \frac{{\rm d}}{{\rm d}x}\log f_{R}(x) = \phi(c/x)-(b+1).$$ Moreover, $\phi(0^{+})=0$ and $\phi(x) \to \infty$ as $x \to \infty$. Consequently, there is a unique $x^*>0$  which satisfies $x^*(1+(1-a)Q(x^*)) = b+1$ and $f_{R}$ has a unique mode $m^* = c/x^*$. Since $0<xQ(x)<1$, we have that $a+b < x^* < b+1$, which is equivalent to the elementary mode bounds
$$\frac{c}{b+1}< m^* < \frac{c}{a+b}.$$ Thus the Stieltjes transform representation provides both uniqueness and localization of the mode.}
\end{remark}

\section{\bf Integral representations and bounds for $S_{a, b, \alpha}(x)$}
\setcounter{equation}{0}

The following theorem establishes a new integral representation for $S_{a,b,\alpha}(x)$, which will be useful in deriving monotonicity properties and bounds for $S_{a,b,\alpha}(x)$.

\begin{theorem}\label{integral representation theorem}
If $a,b,\alpha>0,$ then the survival function $S_{a,b,\alpha}(x)$ admits the following integral representation
\begin{equation}\label{integral_rep_S}
S_{a,b,\alpha}(x) = \frac{1}{\Gamma(a)\Gamma(b)}
\int_{0}^{\infty}
e^{-s}s^{b-1}
\Gamma\left(a,\frac{abx}{\alpha s}\right){\rm d}s.
\end{equation}
\end{theorem}

\begin{proof}[\bf Proof of Theorem \ref{integral representation theorem}]
By using the integral representation of the modified Bessel function \eqref{integral representation for kmu} in \eqref{survival introduction}, we have that
\begin{align*}
S_{a, b, \alpha}(x) &= \frac{2(ab)^{\frac{a+b}{2}}}{\Gamma(a)\Gamma(b)\alpha^{\frac{a+b}{2}}}\int_{x}^{\infty}t^{\frac{a+b}{2}-1}K_{a-b}\left(2\sqrt{\frac{ab}{\alpha}t}\right) {\rm d}t\\
&=\frac{2(ab)^{\frac{a+b}{2}}}{\Gamma(a)\Gamma(b)\alpha^{\frac{a+b}{2}}}\int_{x}^{\infty}t^{\frac{a+b}{2}-1}\cdot\frac{1}{2}\left(\frac{abt}{\alpha}\right)^{\frac{a-b}{2}}\left(\int_{0}^{\infty}e^{-s-\frac{ab}{\alpha s}t}\frac{1}{s^{a-b+1}} {\rm d}s\right) {\rm d}t\\
&=\frac{(ab)^{a}}{\Gamma(a)\Gamma(b)\alpha^a}\int_{0}^{\infty}\frac{e^{-s}}{s^{a-b+1}}\left(\int_{x}^{\infty}t^{a-1}e^{-\frac{ab}{\alpha s}t} {\rm d}t\right) {\rm d}s.
\end{align*}
By using the transformation $T={abt}/{\alpha s} $, we obtain the required integral representation \eqref{integral_rep_S}.
\end{proof}

By using the integral representation \eqref{integral_rep_S} the following theorem establishes that $x \mapsto S_{a, b, \alpha}(x)$ is completely monotone on $(0, \infty)$ and $\alpha \mapsto S_{a, b, \alpha}(x)$ is increasing on $(0, \infty)$.

\begin{theorem}\label{consequence of integral}
Let $a,$ $b,$ $\alpha >0$. The following statements are true.
    \begin{enumerate}
        \item[\bf a.] If $a \leq 1$ or $b \leq 1$, then $x \mapsto S_{a, b, \alpha}(x)$ is completely monotone on $(0, \infty)$.
        \item[\bf b.] The function $\alpha \mapsto S_{a, b, \alpha}(x)$ is increasing on $(0, \infty)$.
    \end{enumerate}
\end{theorem}

\begin{proof}[\bf Proof of Theorem \ref{consequence of integral}]
{\bf a.} From the integral representation \eqref{integral_rep_S} we can easily verify that
$$\frac{{\rm d}}{{\rm d}x}S_{a, b, \alpha}(x) = -\frac{1}{\Gamma(a)\Gamma(b)}\left(\frac{ab}{\alpha}\right)^{a}\int_{0}^{\infty} e^{-s}s^{b-a}\left(x^{a-1}e^{-\frac{ab}{\alpha s}x}\right) {\rm d}s.$$ Note that when $a\leq 1$, the functions $x \mapsto x^{a-1}e^{-\frac{ab}{\alpha s}x}$ is completely monotonic on $(0, \infty)$. Thus, when $a \leq 1$, the function $x \mapsto S_{a, b, \alpha}(x)$ is completely monotone on $(0, \infty)$. Furthermore for the case $b \leq 1$, the same conclusion follows because of the symmetry $S_{a, b, \alpha}(x) = S_{b, a, \alpha}(x)$.

{\bf b.} By differentiating the integral representation \eqref{integral_rep_S} with respect to $\alpha$ leads to
$$ \frac{{\rm d}}{{\rm d}\alpha}S_{a, b, \alpha}(x) = \frac{(abx)^a}{\Gamma(a)\Gamma(b)}\left(\frac{1}{\alpha^{a+1}}\right)\int_{0}^{\infty}e^{-s}s^{b-a-1}e^{-\frac{abx}{\alpha s}} {\rm d}s>0.$$
\end{proof}

The following remark provides an alternative proof of the result established in Theorem \ref{a<1_NWU}.

\begin{remark}
{\em Recall that for $a\leq 1$ or $b\leq 1$, the function $x\mapsto S_{a,b,\alpha}(x)$ is completely monotonic on $(0,\infty)$. Consequently, by using the fact that every completely monotone function on $(0, \infty)$ is log-convex (see for example \cite[Theorem A]{Fi82}), we have that $x\mapsto S_{a,b,\alpha}(x)$ is log-convex on $(0,\infty)$. Thus, in view of Lemma \ref{NBU_theorem_S}, the $K$-distribution possesses the NWU property.}
\end{remark}

By using the integral representation \eqref{integral_rep_S} the following corollary gives a lower bound for the survival function $S_{a, b, \alpha}(x)$.

\begin{corollary}\label{lower bound for survival}
    If $a\geq 1$ and $\alpha,$ $b>0$, then the following inequality is valid
    \begin{equation}
        S_{a, b, \alpha}(x) \geq \frac{2}{\Gamma(a)\Gamma(b)}\left(\frac{abx}{\alpha}\right)^{\frac{a+b-1}{2}}K_{a-b-1}\left(2\sqrt{\frac{abx}{\alpha}}\right)
    \end{equation}
\end{corollary}

\begin{proof}[\bf Proof of Corollary \ref{lower bound for survival}]
By using the following bound for the incomplete gamma function (see \cite[Proposition 2.7]{Pi20})
\begin{equation*}
\Gamma(a, x) \geq x^{a-1} e^{-x},\quad \mbox{for}\quad a\geq 1
\end{equation*}
in the integral representation \eqref{integral_rep_S}, we have that
\begin{eqnarray*}
S_{a, b, \alpha}(x) \geq \frac{1}{\Gamma(a)\Gamma(b)}\left(\frac{abx}{\alpha}\right)^{a-1}\int_{0}^{\infty}e^{-s-\frac{(2\sqrt{abx/\alpha})^2}{4s}} \frac{1}{s^{a-b-1+1}} {\rm d}s.
\end{eqnarray*}
In view of the integral representation of modified Bessel function of second kind \eqref{integral representation for kmu}, we obtain that
\begin{equation*}
S_{a, b, \alpha}(x) \geq \frac{2}{\Gamma(a)\Gamma(b)}\left(\frac{abx}{\alpha}\right)^{a-1}\left(\frac{\alpha}{abx}\right)^{\frac{a-b-1}{2}} K_{a-b-1}\left(2\sqrt{\frac{abx}{\alpha}}\right).
\end{equation*}
Consequently, we arrive at the desired lower bound for $S_{a, b, \alpha}(x)$.
\end{proof}

The following corollary presents a lower bound for $S_{a, b, \alpha}(x)$ in terms of incomplete gamma function.

\begin{corollary}\label{another lower bound survival}
Let $a-b\geq \frac{1}{2}$ and $\alpha>0$. Then the following inequality is true.
\begin{equation}\label{lower_bound_S_2}
S_{a, b, \alpha}(x) > \frac{4^{1-b}\Gamma(a-b)}{2\Gamma(a)\Gamma(b)}\Gamma\left(2b, 2\sqrt{\frac{ab}{\alpha}x}\right).
\end{equation}
\end{corollary}

\begin{proof}[\bf Proof of Corollary \ref{another lower bound survival}]
By using the following bound for the modified Bessel function of second kind (see \cite[Eq. 1.4]{Is90})
\begin{equation}
K_{\mu}(x)> 2^{\mu-1}\Gamma(\mu)x^{-\mu}e^{-x}, \quad \mu\geq 1/2,\quad \mbox{and}\quad x>0
\end{equation}
in $S_{a, b, \alpha}(x)$, we obtain that
\begin{align}
S_{a, b, \alpha}(x) &= \frac{2(ab)^{\frac{a+b}{2}}}{\Gamma(a)\Gamma(b)\alpha^{\frac{a+b}{2}}}\int_{x}^{\infty}t^{\frac{a+b}{2}-1}K_{a-b}\left(2\sqrt{\frac{ab}{\alpha}t}\right) {\rm d}t \nonumber\\
&> \frac{\Gamma(a-b)}{\Gamma(a)\Gamma(b)}\left(\frac{ab}{\alpha}\right)^{b}\int_{x}^{\infty} t^{b-1}e^{-2\sqrt{\frac{ab}{\alpha}t}} {\rm d}t.\label{bound_gamma}
\end{align}
By using the transformation $2\sqrt{{(ab)}/({\alpha}t)} = u$ in the inequality \eqref{bound_gamma} leads to the required inequality \eqref{lower_bound_S_2} for $S_{a, b, \alpha}(x)$.
\end{proof}

The following theorem provides both upper and lower bounds for $S_{a, b, \alpha}(x)$ which involve both modified Bessel function of second kind and incomplete gamma function.

\begin{theorem}\label{bounds_upper_lower_1_2}
Let $\alpha>0$. Then, the following statements are true.
\begin{enumerate}
\item[\bf a.] If $a-b>-\frac{1}{2}$, then
\begin{equation}\label{bounds_lower_1}
S_{a, b, \alpha}(x) \leq \frac{4^{1-a}}{\Gamma(a)\Gamma(b)}\left(\frac{abx}{\alpha}\right)^{\frac{b-a}{2}}e^{2\sqrt{abx/\alpha}}K_{a-b}\left(2\sqrt{abx/\alpha}\right)\Gamma\left(2a, 2\sqrt{abx/\alpha}\right)
\end{equation}
\item[\bf b.] If $a-b>\frac{1}{2}$, then
\begin{equation}  \label{bounds_lower_upper_1}
S_{a, b, \alpha}(x) \geq  \frac{4^{1-b}}{\Gamma(a)\Gamma(b)}\left(\frac{abx}{\alpha}\right)^{\frac{a-b}{2}}e^{2\sqrt{abx/\alpha}}K_{a-b}\left(2\sqrt{abx/\alpha}\right)\Gamma\left(2b, 2\sqrt{abx/\alpha}\right).
\end{equation}
\end{enumerate}
\end{theorem}

\begin{proof}[\bf Proof of Theorem \ref{bounds_upper_lower_1_2}]
Recall that in \cite[Lemma 4]{BBSV21}, it was established that if $\mu>-\frac{1}{2}$, then $x \mapsto e^{x}x^{-\mu}K_{\mu}(x)$ is decreasing on $(0, \infty)$. Consequently, when $x<t$, we have that
\begin{eqnarray*}
e^{2\sqrt{abt/\alpha}}\left(2\sqrt{abt/\alpha}\right)^{b-a}K_{a-b}\left(2\sqrt{abt/\alpha}\right)&\leq&
e^{2\sqrt{abx/\alpha}}\left(2\sqrt{abx/\alpha}\right)^{b-a}K_{a-b}\left(2\sqrt{abx/\alpha}\right)\\
t^{\frac{a+b}{2}-1}K_{a-b}\left(2\sqrt{abt/\alpha}\right)&\leq&e^{2\sqrt{abx/\alpha}}e^{-2\sqrt{abt/\alpha}}x^{\frac{b-a}{2}}t^{a-1}K_{a-b}\left(2\sqrt{abx/\alpha}\right).
\end{eqnarray*}
By using the above inequality in $S_{a, b, \alpha}(x)$, we arrive at
\begin{align}
S_{a, b, \alpha}(x) &= \frac{2(ab)^{\frac{a+b}{2}}}{\Gamma(a)\Gamma(b)\alpha^{\frac{a+b}{2}}}\int_{x}^{\infty}t^{\frac{a+b}{2}-1}K_{a-b}\left(2\sqrt{\frac{ab}{\alpha}t}\right) {\rm d}t\nonumber\\
&\leq \frac{2(ab)^{\frac{a+b}{2}}}{\Gamma(a)\Gamma(b)\alpha^{\frac{a+b}{2}}}e^{2\sqrt{abx/\alpha}}x^{\frac{b-a}{2}}K_{a-b}\left(2\sqrt{abx/\alpha}\right)\int_{x}^{\infty}t^{a-1}e^{-2\sqrt{abt/\alpha}} {\rm d}t.\label{lower_bound_relation_1}
\end{align}
By using the transformation $u=2\sqrt{abt/\alpha}$ in the inequality \eqref{lower_bound_relation_1} leads to the required inequality \eqref{bounds_lower_1}.

Recall also that in \cite[Lemma 4]{BBSV21}, it was established that if $\mu>\frac{1}{2}$, then $x \mapsto e^{x}x^{\mu}K_{\mu}(x)$ is increasing on $(0, \infty)$. Consequently, when $t>x$, we have that
\begin{align*}
e^{2\sqrt{abt/\alpha}}\left(2\sqrt{abt/\alpha}\right)^{a-b}K_{a-b}\left(2\sqrt{abt/\alpha}\right)&\geq
e^{2\sqrt{abx/\alpha}}\left(2\sqrt{abx/\alpha}\right)^{a-b}K_{a-b}\left(2\sqrt{abx/\alpha}\right)\\
t^{\frac{a+b}{2}-1}K_{a-b}\left(2\sqrt{abt/\alpha}\right)&\geq e^{2\sqrt{abx/\alpha}}e^{-2\sqrt{abt/\alpha}}x^{\frac{a-b}{2}}t^{b-1}K_{a-b}\left(2\sqrt{abx/\alpha}\right).
\end{align*}
By using the above inequality in $S_{a, b, \alpha}(x)$, we arrive at
\begin{align}
S_{a, b, \alpha}(x) &= \frac{2(ab)^{\frac{a+b}{2}}}{\Gamma(a)\Gamma(b)\alpha^{\frac{a+b}{2}}}\int_{x}^{\infty}t^{\frac{a+b}{2}-1}K_{a-b}\left(2\sqrt{\frac{ab}{\alpha}t}\right) {\rm d}t\nonumber\\
&\geq \frac{2(ab)^{\frac{a+b}{2}}}{\Gamma(a)\Gamma(b)\alpha^{\frac{a+b}{2}}}e^{2\sqrt{abx/\alpha}}x^{\frac{a-b}{2}}K_{a-b}\left(2\sqrt{abx/\alpha}\right)\int_{x}^{\infty}t^{b-1}e^{-2\sqrt{abt/\alpha}} {\rm d}t.\label{upper_bound_relation_2}
\end{align}
By using the transformation $u=2\sqrt{abt/\alpha}$ in the inequality \eqref{upper_bound_relation_2} leads to the inequality \eqref{bounds_lower_upper_1}.
\end{proof}

Since $$S_{a, b, \alpha}(x) = \int_{x}^{\infty}f_{a, b, \alpha}(t) {\rm d}t = \frac{2(ab)^{\frac{a+b}{2}}}{\Gamma(a)\Gamma(b)\alpha^{\frac{a+b}{2}}}\int_{x}^{\infty}t^{\frac{a+b}{2}-1}K_{a-b}\left(2\sqrt{\frac{ab}{\alpha}t}\right) {\rm d}t$$
the change of variables $2\sqrt{\frac{ab}{\alpha}t} = s$ and in view of $K_{-\nu}(x) = K_{\nu}(x)$, we have that
$$S_{a, b, \alpha}(x) = \frac{2^{2-a-b}}{\Gamma(a)\Gamma(b)}\int_{z}^{\infty}s^{a+b-1}K_{\nu}(s) {\rm d}s,$$ where $z = 2\sqrt{\frac{ab}{\alpha}x}$ and $\nu = |a-b|$. Now, define $\nu_{-} = \min\{\nu, 1/2\}$ and $\nu_{+} = \max \{\nu, 1/2\}$. Then for $r \geq s >0$ we arrive at
\begin{equation}\label{inequality modified ratio}
e^{-(r-s)}\left(\frac{s}{r}\right)^{\nu_{+}}\leq \frac{K_{\nu}(r)}{K_{\nu}(s)} \leq e^{-(r-s)}\left(\frac{s}{r}\right)^{\nu_{-}}.
\end{equation}
For this recall that when $\nu > 1/2$, $x \mapsto e^{x}x^{\nu}K_{\nu}(x)$ is increasing on $(0, \infty)$ and when $\nu \in [0, 1/2)$, the same function is decreasing; whereas the function $x \mapsto e^{x}\sqrt{x}K_{\nu}(x)$ has the opposite monotonicity (see \cite[p. 589]{Ba10}). When $\nu = 1/2$, both statements reduce to the particular case when we have a constant, since $K_{1/2}(s) = \sqrt{{\pi}/{2 s}}\cdot e^{-s}$.

\begin{theorem}\label{bounds_monotone}
    Let $a,$ $b,$ $\alpha,$ $x>0$ and let $z = 2\sqrt{{abx}/{\alpha}}$, $\nu = \mid a-b \mid$, $\nu_{+}$ and $\nu_{-}$ as before. Then, we have the following lower and upper bounds
    $$\frac{2^{2-a-b}}{\Gamma(a)\Gamma(b)}e^{z}K_{\nu}(z)z^{\nu_{+}}\Gamma(a+b-\nu_{+}, z) \leq S_{a, b, \alpha}(x) \leq \frac{2^{2-a-b}}{\Gamma(a)\Gamma(b)}e^{z}K_{\nu}(z)z^{\nu_{-}}\Gamma(a+b-\nu_{-}, z).$$ The inequalities are strict when $\nu \neq 1/2$. If $\nu = 1/2$, then the lower and upper bounds coincide with the survival function.
\end{theorem}

\begin{proof}[\bf Proof of Theorem \ref{bounds_monotone}]
By using the inequality \eqref{inequality modified ratio} for $s=z$ on the ratio of quotient $K_{\nu}(r)/K_{\nu}(s)$, multiplication by $s^{a+b-1}$ and integration yields
$$\int_{z}^{\infty}r^{a+b-1}K_{\nu}(r) {\rm d}r \geq e^{z}K_{\nu}(z)z^{\nu_{+}}\int_{z}^{\infty}r^{a+b-\nu_{+}-1}e^{-r} {\rm d}r = e^{z}K_{\nu}(z)z^{\nu_{+}}\cdot \Gamma(a+b-\nu_{+}, z),$$ which after multiplication by $2^{2-a-b}/(\Gamma(a)\Gamma(b))$ is exactly the stated lower bound. The upper bound follows identically from inequality \eqref{inequality modified ratio} with exponent $\nu_{-}$. When $\nu = 1/2$, the ratio $$\frac{K_{1/2}(r)}{K_{1/2}(z)} = e^{-(r-z)}\left(\frac{z}{r}\right)^{1/2}$$ is exact and hence both integrations implies the equality.
\end{proof}

\begin{remark}
{\em The functional inequalities in \cite{BPV11} implies that $x \mapsto \eta(x) = {xf_{a, b, \alpha}'(x)}/{f_{a, b, \alpha}(x)}$ is decreasing, or equivalently the gamma-gamma distribution is geometrically concave. Now, observe that
$$\eta(x) = \frac{a+b}{2}-1+\frac{1}{2} \frac{zK'_{\nu}(z)}{K_{\nu}(z)}, \quad z = 2\sqrt{{abx}/{\alpha}}, \quad \nu = \mid a-b \mid.$$ On the other hand, if $g_{a, b, \alpha}(x) = \log f_{a,b, \alpha}(e^x)$, then clearly $$g'_{a, b, \alpha}(x) = e^x \cdot f_{a, b, \alpha}'(e^x)/f_{a, b, \alpha}(e^x) = \eta(e^x).$$ Thus, $\eta(u)$ decreasing in $u$ means that $g'_{a, b, \alpha}(x)$ is decreasing in $x$, so $g_{a, b, \alpha}(x)$ is concave. This is precisely what is meant by $f_{a, b, \alpha}$ being geometrically concave. For a concave function, its graph lies below every tangent line, that is
$$g_{a, b, \alpha}(y) \leq g_{a, b, \alpha}(x) + g'_{a, b, \alpha}(x) (y-x).$$ Now, by choosing $x = \log u$ and $y = \log t$, we arrive at
$$\log f_{a, b, \alpha}(t) \leq \log f_{a, b, \alpha}(u) + \eta(u)(\log t-\log u),$$ which immediately implies that $$f_{a, b, \alpha}(t) \leq f_{a, b, \alpha}(u) (t/u)^{\eta(u)}$$ whenever $\eta(u) < -1.$ Integrating the above inequality over $(u, \infty)$ yields
$$S_{a, b, \alpha}(u) = \int_{u}^{\infty}f_{a, b, \alpha}(t) {\rm d}t \leq f_{a, b, \alpha}(u)\int_{u}^{\infty}\left(\frac{t}{u}\right)^{\eta(u)} {\rm d}t.$$ Now, by using the change of variables $r = t/u$, we arrive at
$$S_{a, b, \alpha}(u) \leq uf_{a, b, \alpha}(u)\int_{1}^{\infty}\left(r\right)^{\eta(u)} {\rm d}r = uf_{a, b, \alpha}(u)\left(\frac{-1}{\eta(u)+1}\right)=\frac{uf_{a, b, \alpha}(u)}{-1-\eta(u)},$$ where we suppose that $-1-\eta(u)>0$ and hence we have $\lim_{r \to \infty}{r^{\eta(u)+1}}=0$. The well-known logarithmic derivative inequality in \cite{BPV11}
$$\frac{zK'_{\nu}(z)}{K_{\nu}(z)} < -\sqrt{z^2+\nu^2}, \quad z>0$$ implies that
$$\eta (u)<\frac{a+b}{2}-1-\frac{1}{2}\sqrt{z^2+\nu^2}.$$ Since $z^2+\nu^2-(a+b)^2=4ab\left(\frac{x}{\alpha}-1\right)$, we have $\eta(u)<-1$ whenever $x > \alpha$. Therefore, we arrive at $$S_{a, b, \alpha}(x) < \frac{2xf_{a, b, \alpha}(x)}{\sqrt{4abx/\alpha+(a-b)^2}-(a+b)}, \quad x>\alpha.$$ This upper bound is simpler than that the incomplete gamma function bound and involves only the gamma-gamma density and elementary operations. It is also asymptotically sharp as $x \to \infty$.}
\end{remark}

\section{\bf Series representation for $S_{a, b, \alpha}(x)$}
In this section, we present two series representations for $S_{a,b,\alpha}(x)$, corresponding to the cases $a-b=n+\frac{1}{2}$ and $a-b\neq n+\frac{1}{2}$, where $n \in \{0, 1, 2, \dots\}$. The derivations are based on series representations of the modified Bessel function of the second kind given in \cite[p. 72]{MOS66} and \cite[eq. 6]{Molu}.

\begin{theorem} \label{series representation one}

For $\alpha>0$ the closed form representation  for  $S_{a, b, \alpha}(x)$ when $a-b = n+1/2$ is as follows
$$ S_{a, b, \alpha}(x)=\sqrt{\frac{\pi}{2}} \frac{2^{\frac{5}{2}-2 a}}{\Gamma(a)\Gamma(b)} \sum_{k=0}^{a-b-1/2} \frac{\Gamma(2a-2b-k)}{\Gamma(a-b-k+1/2)\Gamma(k+1)2^{-k}} \Gamma\left(2b+k, 2\sqrt{\frac{ab}{\alpha}x}\right).$$
\end{theorem}

\begin{proof}[\bf Proof of Theorem \ref{series representation one}]
Note that the modified Bessel function of second kind $K_{\mu}(x)$, for $\mu=n+1/2$, $n\in \{0, 1, 2, \dots \},$ has the following representation (see \cite[p. 72]{MOS66})
\begin{equation}\label{kmu_rep_1}
K_{n+\frac{1}{2}}(y)=\sqrt{\dfrac{\pi}{2x}}\dfrac{{\rm e}^{-x}}{x^n}H_n(y),
\end{equation}
where $H_n$ is a monic polynomial of degree $n$ defined as $$H_n(y)=\sum_{k=0}^n\dfrac{\Gamma(2n-k+1)\,y^k}{\Gamma(n-k+1)\,\Gamma(k+1)\,2^{n-k}}.$$
For the case $a-b = n+\frac{1}{2}$, by using the representation \eqref{kmu_rep_1} in $ S_{a, b, \alpha}(x)$, we arrive at
\begin{align*}
S_{a, b, \alpha}(x) =& \frac{2(ab)^{\frac{a+b}{2}}}{\Gamma(a)\Gamma(b)\alpha^{\frac{a+b}{2}}} \int_{x}^{\infty} t^{\frac{a+b}{2}-1} K_{a-b}\left(2\sqrt{\frac{ab}{\alpha}t}\right){\rm d}t \\
=& \sqrt{\frac{\pi}{2}} \frac{2^{2b-2a+\frac{3}{2}}}{\Gamma(a)\Gamma(b)} \left(\frac{ab}{\alpha}\right)^{b}\sum_{k=0}^{a-b-1/2} \frac{\Gamma(2a-2b-k)}{\Gamma(a-b-k+1/2)\Gamma(k+1)2^{-2k}} \left(\sqrt{\frac{ab}{\alpha}}\right)^{k} \int_{x}^{\infty} t^{b-1+\frac{k}{2}} {\rm e}^{-2\sqrt{\frac{abt}{\alpha}}}{\rm d}t \\
=& \sqrt{\frac{\pi}{2}} \frac{2^{\frac{5}{2}-2 a}}{\Gamma(a)\Gamma(b)} \sum_{k=0}^{a-b-1/2} \frac{\Gamma(2a-2b-k)}{\Gamma(a-b-k+1/2)\Gamma(k+1)2^{-k}} \Gamma\left(2b+k, 2\sqrt{\frac{ab}{\alpha}x}\right).
\end{align*}
\end{proof}

\begin{theorem}\label{series representation two}
For $\alpha>0$ the closed form representation  for  $S_{a, b, \alpha}(x)$ when $a-b \neq n+1/2$ is as follows
  \begin{align*}
S_{a,b,\alpha}(x)&=\frac{2^{2-2a}\sqrt{\pi}\Gamma(2a-2b)}{\Gamma(a)\Gamma(b)
\Gamma\left(\frac{1}{2}-a+b\right)}\sum_{n=0}^{\infty}\sum_{k=0}^{n}\frac{(-1)^k
\Gamma\left(\frac{1}{2}+n-a+b\right)L(n,k)}{\Gamma\left(\frac{1}{2}+n+a-b\right)n!}\Gamma\left(2b+k,2\sqrt{\frac{ab}{\alpha}x}\right).
\end{align*}
\end{theorem}

\begin{proof}[\bf Proof of Theorem \ref{series representation two}]
 Note that the modified Bessel function of second kind $K_{\nu}(\beta x)$, for $\nu\notin\{0,\frac{1}{2},\frac{3}{2},\ldots\}$ and $\nu>1$ has the following infinite series representation (see \cite[eq. 6]{Molu})
 \begin{equation}\label{kmu_rep_2}
 K_{\nu}(\beta x)
=e^{-\beta x}\sum_{n=0}^{\infty}\sum_{k=0}^{n}\Lambda(\nu,n,k)(\beta x)^{k-\nu},
\end{equation}
where $\Lambda(\nu,n,k)=\dfrac{(-1)^{k}\sqrt{\pi}\Gamma(2\nu)\Gamma\left(\frac{1}{2}+n-\nu\right)L(n,k)}{2^{\nu-k}\Gamma\left(\frac{1}{2}-\nu\right)\Gamma\left(\frac{1}{2}+n+\nu\right)n!},$ and the Lah numbers are defined as $$L(n,k)=\binom{n-1}{k-1}\dfrac{n!}{k!}$$ with the conventions that $L(0,0)=1$, $L(n,0)=0$, and $L(n,1)=n!$. By using the series representation \eqref{kmu_rep_2} in $S_{a, b, \alpha}(x)$, we arrive at
\begin{align*}
S_{a,b,\alpha}(x)=&\frac{2}{\Gamma(a)\Gamma(b)}\left(\frac{ab}{\alpha}\right)^{\frac{a+b}{2}}\int_x^\infty t^{\frac{a+b}{2}-1}e^{-2\sqrt{\frac{ab}{\alpha}t}}
\sum_{n=0}^{\infty}\sum_{k=0}^{n}\Lambda(a-b,n,k)
\left(2\sqrt{\frac{ab}{\alpha}t}\right)^{k-a+b}\,{\rm d}t\\
=&\frac{2^{b-a+1}}{\Gamma(a)\Gamma(b)}\left(\frac{ab}{\alpha}\right)^{b}\sum_{n=0}^{\infty}\sum_{k=0}^{n}\Lambda(a-b,n,k)\left(\frac{ab}{\alpha}\right)^{\frac{k}{2}}
\int_x^\infty t^{\frac{a+b}{2}-1+\frac{k-a+b}{2}}e^{-2\sqrt{\frac{ab}{\alpha}t}}
\,{\rm d}t\\
=&\frac{2^{b-a+1}}{\Gamma(a)\Gamma(b)}\left(\frac{ab}{\alpha}\right)^{b}\sum_{n=0}^{\infty}\sum_{k=0}^{n}\Lambda(a-b,n,k) \left(\frac{ab}{\alpha}\right)^{\frac{k}{2}}
\int_x^\infty t^{b+\frac{k}{2}-1}e^{-2\sqrt{\frac{ab}{\alpha}t}}{\rm d}t.
\end{align*}
Note that
\begin{align*}
\int_x^\infty t^{b-1+\frac{k}{2}}
e^{-2\sqrt{\frac{ab}{\alpha}t}}\,{\rm d}t
&=2\left(2\sqrt{\frac{ab}{\alpha}}\right)^{-2b-k}
\int_{2\sqrt{\frac{ab}{\alpha}x}}^\infty
u^{2b+k-1}e^{-u}\,{\rm d}u\\
&=2\left(2\sqrt{\frac{ab}{\alpha}}\right)^{-2b-k}
\Gamma\left(2b+k,2\sqrt{\frac{ab}{\alpha}x}\right).
\end{align*}
Consequently, we arrive at
\begin{align*}
S_{a,b,\alpha}(x)&=\frac{2^{2-a-b}}{\Gamma(a)\Gamma(b)}\sum_{n=0}^{\infty}\sum_{k=0}^{n}\Lambda(a-b,n,k)
\Gamma\left(2b+k,2\sqrt{\frac{ab}{\alpha}x}\right).
\end{align*}
By substituting the expression for $\Lambda(a-b,n,k)$ in the above equation leads to the desired result.
\end{proof}

\section{\bf Summary}

In this paper, we have shown that the gamma-gamma distribution possesses the NWU property for certain ranges of the parameters, whereas in the remaining parameter region, it possesses neither the NBU nor the NWU property. Our proofs rely on concepts from the theory of \textit{stochastic orders}. By exploiting the representation of the gamma-gamma distribution as the product of two independent gamma random variables, we established alternative proofs of these results. Moreover, since the beta-prime distribution can be represented as the quotient of two independent gamma random variables, we established analogous results concerning its NBU and NWU properties. We also established a general result showing that the product of two independent positive random variables possesses the NWU property whenever at least one of the random variables possesses the NWU property. We obtained a similar result for the DFR property. Furthermore, we represented the survival function of the gamma-gamma distribution as the Laplace transform of a probability distribution belonging to the class of generalized gamma convolutions, thereby establishing that it is hyperbolically completely monotone. We also proved that the failure rate function of the gamma-gamma distribution is completely monotone and that its Mills ratio is a Bernstein function. Finally, we derived some bounds and series representations for the survival function of the gamma-gamma distribution. We hope that the results derived in this paper may be interesting and useful for the electrical engineering people and also for the mathematical community.

\end{document}